\documentclass[11pt]{article}
\usepackage[utf8]{inputenc}
\usepackage[T1]{fontenc}

\usepackage[top=1in, bottom=1in, left=1.25in, right=1.25in]{geometry}

\usepackage{amsmath,amssymb,amsthm,amsfonts,mathtools}
\usepackage{color}
\usepackage{multirow}
\usepackage[absolute,overlay]{textpos}

\usepackage[sort,compress]{cite}
\usepackage{graphicx}
\usepackage{ragged2e}
\usepackage{tabto}

\usepackage{lmodern}

\usepackage[sf,bf,medium]{titlesec}
\usepackage{setspace}
\usepackage{bm}

\usepackage{pgfplots}
\pgfplotsset{compat=1.18} 

\usepackage[font=footnotesize]{caption}
\usepackage[font=footnotesize]{subcaption}

\usepackage{booktabs}
\usepackage{siunitx}
\usepackage{isomath}

\usepackage[plainpages=false, colorlinks=true, citecolor=blue, filecolor=blue,     
linkcolor=blue, urlcolor=blue]{hyperref}
\usepackage{cleveref}
\usepackage{enumitem}

\usepackage{placeins}
\makeatletter
\AtBeginDocument{%
  \expandafter\renewcommand\expandafter\subsection\expandafter
    {\expandafter\@fb@secFB\subsection}%
  \newcommand\@fb@secFB{\FloatBarrier
    \gdef\@fb@afterHHook{\@fb@topbarrier \gdef\@fb@afterHHook{}}}%
  \g@addto@macro\@afterheading{\@fb@afterHHook}%
  \gdef\@fb@afterHHook{}%
}
\makeatother

\usepackage{float}
\usepackage{color,graphicx,epsfig,fancyhdr}

\usepackage[T1]{fontenc}
\usepackage{comment}
\usepackage{commath}
\usepackage{bbm}
\usepackage{listings}
\usepackage{mathrsfs}

\newtheorem{theorem}{Theorem}[section]

\newtheorem{lemma}{Lemma}[section]

\newcommand{\Ione}{\mathcal{I}_\mathrm{DSM\text{-}1}}
\newcommand{\Itwo}{\mathcal{I}_\mathrm{DSM\text{-}2}}

\DeclareSymbolFont{matha}{OML}{txmi}{m}{it}
\DeclareMathSymbol{\varv}{\mathord}{matha}{118}
\newcommand{\dis}{\operatorname{dist}}
\newcommand{\dd}{\mathrm{d}}

\title{
On Sampling Methods for Inverse Biharmonic Scattering Problems in Supported Plates }
\author{
Carlos Borges\thanks{School of Data, Mathematical, and Statistical Sciences, University of Central Florida, Orlando, FL 32816 \textit{Email:Carlos.Borges@ucf.edu} }
\and  
Rafael Ceja-Ayala\thanks{School of Mathematical and Statistical Sciences, Arizona State University, Tempe, AZ 85287\textit{Email:rcejaaya@asu.edu} }
\and
Peter Nekrasov\thanks{Committee on Computational and Applied Mathematics, University of Chicago, Chicago, IL 60637 \textit{Email:pn3@uchicago.edu} }}

\begin{document}

\maketitle

\numberwithin{equation}{section}

\begin{abstract}
We study the inverse problem of qualitatively recovering a supported cavity in a thin elastic plate governed by the flexural (biharmonic) wave equation, using far-field pattern measurements. We derive a reciprocity principle and a factorization of the far-field operator for the supported plate boundary conditions, and we analyze its range properties to justify both the linear sampling method (LSM) and the direct sampling method (DSM). Numerical experiments assess the performance of LSM and DSM under noise, a limited amount of data, multiple scattering, and variations in the Poisson's ratio. The results show that both methods robustly recover the obstacle’s location and convex hull, with DSM offering improved stability and reduced computational cost.
\end{abstract}


\section{Introduction}

In recent years, biharmonic problems have garnered significant attention due to their applications in the analysis of resonances in bridges and large structures~\cite{drabek2003nonlinear, PhysRevB.73.064301, lindsay2018boundary, lindsay2015transition}, simulation of acoustic ``black holes''~\cite{pelat2020acoustic}, development of topological waves in elastic systems~\cite{ganti2020topological,mousavi2015topologically}, nondestructive testing of aerospace composites~\cite{memmolo2018guided}, design of platonic crystals and metamaterials ~\cite{elasticcloak, farhat2012broadband, farhat2009cloaking, gao2018theoretical}, and the study of ice shelf flexure and ice-ocean interactions \cite{nekrasov2023ocean, askham2025surface, askham2025integral, nekrasov2025boundary}.

We consider the inverse problem of reconstructing the shape of a supported cavity in a thin elastic plate from far-field scattering data. Letting the cavity occupy a compact domain $D\subset\mathbb{R}^2$ with smooth boundary $\partial D$, the vertical displacement $u$ of the  plate satisfies the following time-harmonic equation:
\begin{equation}
\begin{cases}
    \Delta^2 u - k^4 u = 0, & \text{in } \mathbb{R}^2 \setminus \overline{D}, \\
    u = 0,  & \text{on } \partial D, \\
   \nu \Delta u + (1-\nu) \partial_n^2 u = 0, & \text{on } \partial D,
\end{cases} \label{eq:BVP}
\end{equation}
where $\Delta$ denotes the two-dimensional Laplacian, $\partial_n$ denotes the directional derivative with respect to the normal vector $n$, $\nu \in [-1,0.5]$ is the Poisson's ratio of the plate, and $k > 0$ is the wavenumber, which depends on the thickness of the plate and the time frequency of the incident field. The first boundary condition corresponds to the displacement of the plate, while the second boundary condition corresponds to the bending moment. These boundary conditions are frequently used to model plates that are supported by internal columns and rods \cite{seide1958,zhao2002plate} or for ice shelves supported by land at the grounding line \cite{MEYLAND2021}.

Previous work on biharmonic inverse scattering has focused on the recovery of volumetric potentials from boundary and far-field data \cite{logestimates,ikehata1991special,krupchyk2012determining,yang2014determining,liu2024stability, inversebackscattering} or the reconstruction of obstacles with \textit{clamped plate} boundary conditions \cite{HLL-biharmonic,guo2024direct,bourgeois2020linear, optiforinversebiha}. Many studies have focused on the global uniqueness problem for the recovery of potentials or other material parameters associated with biharmonic operators \cite{firstorderper,krupchyk2012determining,partialdatainverse}. Related work has also been done to understand the reconstruction of potentials and internal sources using passive measurements generated by unknown sources \cite{passivemeasure,chang2024analysis}. 

In the present study, we consider methods for determining the location and shape of a supported cavity from far-field measurements of flexural waves. In particular, we consider two sampling methods: the \emph{linear sampling method} (LSM) and the \emph{direct sampling method} (DSM). The LSM reconstructs the cavity boundary by solving a far-field equation for synthetic incident fields, while the DSM offers a computationally efficient alternative via the application of the far-field operator. Our analysis provides a rigorous justification of both the LSM and DSM by examining both the range and reciprocity of the far-field operator for the supported plate problem. Numerical simulations further demonstrate that both methods can accurately and robustly recover the shape and position of the supported cavity, even with a limited amount of data and noisy measurements.

We present several new analytical and computational results for the inverse biharmonic scattering problem. First, we present a reciprocity relation that extends the clamped plate reciprocity result of \cite{harris2026sampling} to the supported plate setting.
Next, we show that the well-established factorization of the far-field operator holds for supported plates, providing a rigorous foundation for the LSM. We also introduce a factorization of the data-to-pattern operator based on the boundary-integral formulations of \cite{farkas,nekrasov2025boundary}, allowing us to study the properties of the DSM. Finally, we perform a comprehensive numerical study of the frequency dependence and robustness to noise and other parameters of the LSM and DSM.

The paper is organized as follows. In Section~\ref{sec:direct-prob-LDSM}, we formulate the forward problem and provide definitions of the fundamental solution and far-field pattern. Section~\ref{sec:LSM} focuses on the linear sampling method, establishing the reciprocity principle and other results about the far-field operator. In Section~\ref{sec:DSM}, we focus on the direct sampling method by deriving a factorization of the data-to-pattern operator and describing the indicator function. Section~\ref{sec:num_results} reports numerical experiments that illustrate the performance and highlight the limitations of the methods across a range of scenarios. Finally, Section~\ref{sec:conclusions} offers concluding remarks and directions for future work.

\section{Problem setup}\label{sec:direct-prob-LDSM}

Let $u$ be the sum of a known incident field $u^i$ and an unknown scattered field $u^s$. If $u^i$ satisfies the homogeneous equation in $\mathbb{R}^2 \setminus \overline{D}$, we have that $u^s$ satisfies 
\begin{equation}
\begin{cases}
    \Delta^2 u^s - k^4 u^s = 0, & \text{in } \mathbb{R}^2 \setminus \overline{D}, \\
    u^s = -u^i ,  & \text{on } \partial D,  \\
     M[u^s] = -M[u^i]  , & \text{on } \partial D ,
    \label{direct1}
\end{cases}
\end{equation}
where the bending moment operator $M$ is defined by
\begin{equation}\label{eq:bendingmom}
M[u] = \nu \Delta u + (1-\nu) \partial_n^2 u , 
\end{equation}
and $u^s$ satisfies the  Sommerfeld radiation condition at infinity:
\begin{equation}
{\partial_r u^s} - i ku^s = o ( r^{-1/2} ) , \quad \text{ as } \quad r \rightarrow \infty . \label{direct2}
\end{equation}
 Let  $\Phi(x,y)$ denote the fundamental solution of the biharmonic wave equation satisfying $$\Delta^2\Phi(x,y)-k^4\Phi(x,y)=\delta(x-y), \quad \text{in} \hspace{.2cm} \mathbb{R}^2 ,$$
along with the Sommerfeld radiation condition \eqref{direct2}. The formula for $\Phi$ is given by:
\begin{equation*}
    \Phi(x,y)
      = \frac{i}{8k^2}
        \bigl( H^{(1)}_0(k|x-y|) - H^{(1)}_0(ik|x-y|)\bigr),
\end{equation*}
where $H^{(1)}_0$ is the zeroth-order Hankel function of the first kind. Note that because $H_0^{(1)}(ik|x-y|)$ decays exponentially away from any source(s), the asymptotic behavior of the fundamental solution is similar to that of $H_0^{(1)}(k|x-y|)$:

$$ \Phi (x , y )= \frac{e^{\text{i}k|x|}}{\sqrt{|x|}}\left( \Phi^{\infty}(\hat{x},y)+\mathcal{O}\left(\frac{1}{|x|}\right) \right) \hspace{.3cm}\text{as}\hspace{.3cm}|x|\to \infty\,, $$
where $\hat{x} := \frac{x}{|x|} \in \mathbb{S}^1$ and  $\Phi^\infty(\hat{x},y)$ denotes the far-field pattern of the fundamental solution, given by:
\begin{equation}
    \Phi^{\infty}(\hat{x},y)
      = -\frac{1}{2k^2}
        \frac{e^{i\pi/4}}{\sqrt{8\pi k}}
        e^{-ik \hat{x} \cdot  y} \,,  \label{eq:Phi_ff} \end{equation}
The scattered field $u^s$ admits a similar expansion at infinity:
\begin{equation}
u^s(x)= \frac{e^{\text{i}k|x|}}{\sqrt{|x|}}\left\{u^{\infty}(\hat{x})+\mathcal{O}\left(\frac{1}{|x|}\right) \right\}\hspace{.3cm}\text{as}\hspace{.3cm}|x|\longrightarrow \infty\,, \label{far-fieldpattern}
\end{equation}
where $u^\infty(\hat{x})$ denotes the far-field pattern corresponding to the scattered field $u^s$. In this paper, we will often use $u^s(x,d)$ to denote the solution of \eqref{direct1} for an incident plane wave in the direction $d \in \mathbb{S}^1$. Similarly, $u^\infty(\hat{x},d)$ represents the far-field pattern of $u^s(\hat{x},d)$.  

Finally, we define the far-field operator $\mathcal{F} : L^2(\mathbb{S}^1) \to L^2(\mathbb{S}^1)$ as:
\begin{equation}
    \mathcal{F}[g](\hat{x}) = \int_{\mathbb{S}^1} u^{\infty}(\hat{x}, d) g(d) \, \mathrm{d}s(d).
    \label{far-field}
\end{equation}
This operator will play a prominent role in both the linear sampling and direct sampling methods. 


In the next two sections, we introduce and analyze the sampling methods considered in this paper. We start with the \emph{linear sampling method} (\Cref{sec:LSM}) followed by the \emph{direct sampling method} (\Cref{sec:DSM}). The fundamental principle of both sampling methods is to determine if a given sampling point $z \in \mathbb{R}^2$ lies inside or outside the domain~$D$ by evaluating a suitable indicator function. For the linear sampling method, this indicator function is given by the solution to a particular integral equation which becomes unbounded for points not in $D$. Meanwhile, the direct sampling method uses an indicator function that comes from applying the far-field operator directly to the fundamental solution. As we will see, this indicator function is more stable to evaluate and leads to better numerical stability. Using either of these indicator functions, one can infer an approximation of the obstacle for the supported plate problem.

\section{Linear sampling method}\label{sec:LSM}

The goal of the linear sampling method (LSM) is to find a function $g_z \in L^2(\mathbb{S}^1)$ for each $z \in \mathbb{R}^2$ such that 
\begin{equation}
    \mathcal{F} [g_z](\hat{x}) = \phi_z(\hat{x}), 
    \label{farfield}
\end{equation}
where $\phi_z(\hat{x}) := \Phi^{\infty}(\hat{x},z)$ is the far-field pattern of the fundamental solution. For the LSM, the $L^2$ norm of $g_z$ acts as an indicator function for the obstacle. In order to see this, we present the theoretical justification behind the linear sampling method (LSM). We begin by establishing a reciprocity principle for the biharmonic scattering problem with supported plate boundary conditions in \Cref{sec:reciprocity}. Then, in \Cref{sec:factorization} we extend the factorization of the far-field  operator to  the supported plate and prove that it is compact, injective, and has dense range, key components in deriving the LSM for our problem. Finally, we present the result that shows that the value of the indicator goes to infinity for sampling points outside the obstacle and remains bounded for points inside.

\subsection{Reciprocity principle}\label{sec:reciprocity}
In this section, we derive a reciprocity relation for the solution $u^s$ to the boundary value problem \eqref{direct1}-\eqref{direct2}. Similar reciprocity principles are well known for acoustic scattering \cite{coltonkress} and for the clamped plate problem \cite{harris2026sampling} but have not been studied for the supported plate problem. These identities are used in inverse scattering to study the corresponding far--field \cite{HLL-biharmonic} or near--field \cite{bourgeois2020linear} operators which are used in qualitative methods for recovering the obstacle.

\begin{theorem}\label{secondreciprocity}
Let $u^{\infty}(\hat{x}, d)$ denote the far-field pattern defined in \eqref{far-fieldpattern} for the supported plate problem \eqref{direct1} with incident field $u^{i} = e^{ikx \cdot d}$, where $d \in \mathbb{S}^1$. Then $u^{\infty}(\hat{x}, d)$ satisfies the reciprocity relation:
\begin{equation*}
    u^{\infty}(\hat{x}, d) = u^{\infty}(-d, -\hat{x}).
\end{equation*}
\end{theorem}

\begin{proof}

We begin by noting the standard integration by parts formula \cite{HsWe2021}:
\begin{equation}
    \int_{\mathbb{R}^2 \setminus D} {v} \Delta^2 u  -u\Delta^2{v} \, \text{d} x  
= \int_{\partial D} \big( v N[ u] -  u N[v]  - \partial_n v M[u]  + \partial_n u M[v]  \big) \, \mathrm{d}s  \label{eq:IBP}
\end{equation}
where $M$ is the bending moment operator defined in \eqref{eq:bendingmom} and $N$ is the shear stress operator given by:
\begin{equation*}
    N[u] := \partial_{n} \Delta u + (1-\nu) \frac{\mathrm{d} \; }{\mathrm{d} s} \partial_{n} \partial_{\tau} u , 
\end{equation*}
where $\partial_\tau$ denotes the tangential derivative and $\frac{\mathrm{d}}{\mathrm{d} s}$ denotes the arc-length derivative. The proof proceeds as in \cite{harris2026sampling}. Letting $u = \Phi(x,y)$ and $v = u^s(x,d)$, and taking the limit as $|x|\to \infty$, one obtains the following expression for the far-field pattern:
\begin{align*}
   u^\infty(\hat{x},d) 
&=  \int_{\partial D}  \big( u^s(y,d) N [u^i(\cdot,-\hat{x})](y) -  u^i(y,-\hat{x}) N [u^s(\cdot,d)](y) \big) \text{d} s_y \\
&\quad - \int_{\partial D} \big( \partial_{n_y} u^s(y,d) M[u^i(\cdot,-\hat{x})](y)  - \partial_{n_y} u^i(y,-\hat{x}) M [u^s(\cdot,d)](y) \big) \text{d} s_y .
\end{align*}
where $u^{i}(y,-\hat{x}) = e^{-ik\hat{x}\cdot y}$. A similar formula holds for $u^\infty(-d,-\hat{x})$. Subtracting these two expressions and using the fact that $u = u^i + u^s$ leads to the following:
{\small
\begin{align*}
     u^\infty(\hat{x},d) - u^\infty(-d,-\hat{x}) &= \int_{\partial D} \big( u (y,-\hat{x}) N [u (\cdot,d)](y) -  u(y,d) N[u(\cdot,-\hat{x})](y) \big) \, \mathrm{d}s_y  \\ &  \quad + \int_{\partial D} \big( - \partial_{n_y} u(y,-\hat{x}) M [u (\cdot,d)](y)  + \partial_{n_y} u (y,d) M [u(\cdot,-\hat{x})](y)  \big) \, \mathrm{d}s_y
\end{align*}
}
where $u(y,d)$ and $u(y,-\hat{x})$ are the solutions to \eqref{eq:BVP} with incident directions $d$ and $-\hat{x} \in \mathbb{S}^1$, respectively. Applying the boundary conditions on $\partial D$ leads to the desired result.
\end{proof} 

Note that this reciprocity result also holds for the well-known \textit{clamped} and \textit{free plate} boundary conditions. With this result, we are ready to analyze the operator $\mathcal{F}$ given in \eqref{far-field} and obtain its factorization.

\subsection{Factorization of the far-field operator}\label{sec:factorization}

To analyze the solvability and stability of equation \eqref{farfield}, we examine a factorization of the far-field operator~$\mathcal{F}$. 
First, we define the \textit{Herglotz wave operator} $ \mathcal{H} : L^2(\mathbb{S}^1) \to H^{3/2}(\partial D)\times H^{-1/2}(\partial D)$:
\begin{equation}
    \mathcal{H}[g] = 
    -\begin{pmatrix}
        v_g \\
        \nu \Delta v_g + (1-\nu) \partial_n^2 v_g
    \end{pmatrix} \bigg|_{\partial D} ,
    \label{Hoperator}
\end{equation}
where $v_g$ is the Herglotz wave function, defined as
\begin{equation}
    v_g(x) = \int_{\mathbb{S}^1} e^{ik x\cdot d}\, g(d)\, \text{d}s(d) .
    \label{Herglotz}
\end{equation} 
We also define the \emph{data-to-pattern operator} $ \mathcal{G} : H^{3/2}(\partial D) \times H^{-1/2}(\partial D) \to L^2(\mathbb{S}^1)$ which maps supported plate boundary data to the far-field pattern of the scattered field $u^s$:
\begin{equation}
    \mathcal{G}[ (f_1, f_2)^T ](\hat{x}) = u^\infty(\hat{x}),
    \label{Goperator}
\end{equation}
In \cite{biwellposed}, solutions to the boundary value problem \eqref{direct1}-\eqref{direct2} were shown to be unique for $u^s \in H^2_{\text{loc}}(\mathbb{R}^2 \setminus \overline{D})$ and for boundary data $(f_1,f_2)\in H^{3/2}(\partial D)\times H^{-1/2}(\partial D)$. Therefore, the data-to-pattern operator $\mathcal{G}$ is well-defined. Then, by the linearity of the boundary value problem, the far-field operator $\mathcal{F}$ can be written as  
\begin{equation}
    \mathcal{F} = \mathcal{G} \mathcal{H}.
    \label{Ffactorization}
\end{equation}
 This decomposition allows us to describe the existence of solutions to \eqref{farfield} in terms of the range of the operator~$\mathcal{G}$. We now present a lemma that establishes a necessary condition for the existence of a solution to \eqref{farfield}: 

\begin{lemma}
\label{RangeG}
Let $\mathcal{G}$ be the data-to-pattern operator defined in \eqref{Goperator}. 
Then $
    \phi_z \in \mathrm{Range}(\mathcal{G})
     $ if and only if $
    z \in D$.
\end{lemma}

The proof of this lemma follows as in  \cite{harris2026sampling} and we omit it here for succinctness. Due to the factorization of $\mathcal{F}$, we have $\mathrm{Range}(\mathcal{F}) \subset \mathrm{Range}(\mathcal{G})$. It follows that $\phi_z \in \mathrm{Range}(\mathcal{F})$ only if $z \in D$. We also have the following lemma for the Herglotz wave operator $\mathcal{H}$:

\begin{lemma} \label{lem:Herglotzcompactness}
The Herglotz wave operator 
$\mathcal{H} : L^2(\mathbb{S}^1) \to H^{3/2}(\partial D) \times H^{-1/2}(\partial D)$ 
defined by \eqref{Hoperator} is compact and injective.
\end{lemma}

\begin{proof}
The proof follows as in \cite{harris2026sampling} with minor modifications. First, to show compactness, observe that the corresponding Herglotz wave functions are smooth solutions of the Helmholtz equation in $\mathbb{R}^2$. 
Hence, we have $v_g \in H^p_{\mathrm{loc}}(\mathbb{R}^2)$ for any $p > 0$ and $\mathcal{H}[g] \in H^{p-1/2}(\partial D) \times H^{p-5/2}(\partial D)$ for any $p > 0$. The compactness of $\mathcal{H}$ into $ H^{3/2}(\partial D) \times H^{-1/2}(\partial D) $ then follows directly from standard Sobolev embedding theorems. To prove injectivity, assume that $\mathcal{H} [g] = 0$. 
This implies that the associated Herglotz wave function satisfies
$v_g = 0$ and $ M[v_g] =-\nu k^2 v_g+(1-\nu)\partial_n^2v_g=(1-\nu)\partial_n^2v_g = 0$ on $\partial D$. We now make use of the standard decomposition of the Laplacian in $\mathbb{R}^2$ restricted to a smooth boundary,
\begin{equation}
\label{decomposition_lap}
\Delta = \partial_n^2 + \kappa\,\partial_n + \Delta_{\mathrm{surf}},
\end{equation}
where $\kappa$ denotes the curvature of $\partial D$ and $\Delta_{\mathrm{surf}}$ is the surface Laplacian; see \cite{Reilly1977}. Since $v_g$ vanishes identically on $\partial D$, it follows that $\Delta_{\mathrm{surf}} v_g = 0$ on $\partial D$. 
Combining this with the condition $\partial_n^2 v_g = 0$ and the decomposition \eqref{decomposition_lap}, we have
$\partial_n v_g = 0$ on $\partial D$. By the unique continuation property for solutions of the Helmholtz equation, it follows that $v_g = 0$ in $\mathbb{R} \setminus \overline{D}$. Consequently, $g = 0$, establishing the injectivity of $\mathcal{H}$.
\end{proof}

Next, we present a result concerning a key analytical property of the far-field operator. Combining the previous lemma with the reciprocity principle \Cref{secondreciprocity}, we arrive at the following result:

\begin{theorem}
\label{Foperatorproperties}
Assume that $k>0$ is not a supported plate transmission eigenvalue, defined as $k$ such that there exists a nontrivial $(p, q) \in H^1(\mathbb{R}^2 \setminus \overline{D}) \times H^1(D)$ satisfying  
\begin{equation}
\begin{cases}
    \Delta p - k^2 p = 0, & \quad \mathrm{ in } \ \mathbb{R}^2 \setminus \overline{D}, \\ 
    \Delta q + k^2 q = 0, & \quad \mathrm{ in } \ D, \\ 
    p + q = 0,  & \quad \mathrm{ on } \ \partial D, \\
   M[p + q]  = 0, & \quad \mathrm{ on } \ \partial D, \\    
\end{cases}
\label{eigenvalueproblem}
\end{equation}
where $p$ decays exponentially as $|x| \to \infty$.    
Then the far-field operator 
$\mathcal{F} : L^2(\mathbb{S}^1) \to L^2(\mathbb{S}^1)$ 
is compact, injective, and has dense range.
\end{theorem}

\begin{proof}
The compactness of $\mathcal{F}$ follows directly from \Cref{lem:Herglotzcompactness}. The proof of injectivity follows as in \cite{harris2026sampling} but we provide a sketch of it here for completeness. Suppose $\mathcal{F}[g] = 0$. Recall that $\mathcal{F}[g](\hat{x})$ is the far-field pattern of the solution $u^s$ for an incident field $u^i = v_g$. For the far-field pattern to be zero, $u^s$ must satisfy the modified Helmholtz equation in $\mathbb{R}^2 \setminus \overline{D}$. Meanwhile, $v_g$ satisfies the Helmholtz equation in $D$. This implies that $(u^s, v_g)$ solves the supported plate transmission eigenvalue problem, which by assumption can only be true if $u^s = v_g = 0$. Therefore, $g = 0$. Finally, to see that $\mathcal{F}$ has dense range, we can use the reciprocity relation to write the adjoint of $\mathcal{F}$ as in \cite{harris2026sampling}:
$$ \mathcal{F}^*[g](\hat{x}) = \int_{\mathbb{S}^1} \overline{u^\infty (-\hat{x},-d)} g(d) \dd s(d) $$
$\mathcal{F}^*$ must also be injective by the same argument as above, and therefore $\overline{\textrm{Range}(\mathcal{F}) } = L^2(\mathbb{S}^1)$. 
\end{proof}

Given the result in Theorem~\ref{Foperatorproperties}, it becomes clear that due to the compactness of the far-field operator $\mathcal{F}$, problem \eqref{farfield} is ill-posed. To mitigate the ill-posedness of the problem, we need to apply some regularization method \cite{engl1996regularization}. In this paper, we will use Tikhonov regularization\, where the problem being solved becomes the minimization of 
\begin{equation}\label{eq:farfieldeq_regularized}
    \| \mathcal{F} [g_z] - \phi_z \|_{L^2(\mathbb{S}^1)}^2 +\alpha \|g_z \|_{L^2(\mathbb{S}^1)}^2,
\end{equation}
where $\alpha>0$ is the regularization parameter. In this case, we will define $g_z^{\alpha}$ as the minimizer of \eqref{eq:farfieldeq_regularized}.

The following result describes the behavior of the indicator function in the LSM as the regularization parameter $\alpha$ goes to zero. In essence, it shows that the approximate solution to the far-field equation \eqref{farfield} remains bounded when the sampling point lies within the scatterer~$D$, which serves as a practical computational criterion for reconstructing the scatterer from the far-field data. 

\begin{theorem}
\label{indicator}
Assume that the wavenumber $k$ is not an eigenvalue of the supported transmission eigenvalue problem given in \eqref{eigenvalueproblem}. Then, for any $z \in \mathbb{R}^2 \setminus \overline{D}$ and a family of functions $\{ g_z^{\alpha} \} \subset L^2(\mathbb{S}^1)$ satisfying
\begin{equation} \label{eq:Flimit}
    \lim_{\alpha \to 0} 
    \| \mathcal{F} [g_z^{\alpha}] - \phi_z \|_{L^2(\mathbb{S}^1)} = 0,
\end{equation}
it follows that
\begin{equation*}
    \lim_{\alpha \to 0} 
    \| g_z^{\alpha} \|_{L^2(\mathbb{S}^1)} = \infty.
\end{equation*}
\end{theorem}

\begin{proof}
The proof follows by an argument similar to \cite{harris2026sampling}; we outline the argument here for completeness and readability. Because $k$ is not a supported transmission eigenvalue, $\mathcal{F}$ is injective, compact, and has dense range. In particular, we know that for $\phi_z \in L^2(\mathbb{S}^1)$ such a family $\{g_z^\alpha\}$ satisfying \eqref{eq:Flimit} exists and can be given as the family of functions that minimizes \eqref{eq:farfieldeq_regularized}.

    Now, suppose by contradiction that $\| g_z^\alpha \|_{L^2(\mathbb{S}^1)}$ remains bounded as $\alpha \to 0$ for some $z \in \mathbb{R}^2 \setminus \overline{D}$. Then, there exists a subsequence $\{ g_z^{\alpha_i} \}^\infty_{i=0} $ such that $ g_z^{\alpha_i} \rightharpoonup g_z $ weakly in  $L^2(\mathbb{S}^1)$ as $i \to \infty$. Since $\mathcal{F}$ is compact, $\lim_{i \to \infty} \mathcal{F}[g_z^{\alpha_i}] = \mathcal{F}[g_z]$ strongly in $L^2(\mathbb{S}^1)$. This implies that there exists a $g_z \in L^2 (\mathbb{S}^1)$ such that $  \mathcal{F}[g_z] = \phi_z $
    for $z \in \mathbb{R}^2 \setminus \overline{D}$. However, this is a contradiction of \Cref{RangeG}, therefore $g_z^\alpha$ is unbounded.
\end{proof}

Finally, the LSM indicator function is defined as
\begin{equation}\label{eq:lsm-indicator}
\mathcal{I}_{\mathrm{LSM}}(z)
=
\frac{1}{\|g_z^\alpha\|_{L^2(S^1)}}.
\end{equation}
By Theorem~\ref{indicator}, $\|g_z^\alpha\|_{L^2(S^1)}$
blows up as $\alpha \to 0$ whenever $z \notin D$, and therefore
the indicator function tends to zero for sampling points $z$ outside the cavity. For $z \in D$, approximate solutions with bounded norm exist, but no uniform
lower bound for $\mathcal{I}_{\mathrm{LSM}}(z)$ can be guaranteed. Thus, the LSM characterizes the cavity $D$ through the growth
behavior of the regularized solutions $g_z^\alpha$.

\section{Direct sampling method}\label{sec:DSM}

While the linear sampling method provides one approach for the qualitative reconstruction of obstacles, it requires solving a large number of ill-posed equations. To address this limitation, we introduce the direct sampling method, which achieves similar reconstruction capabilities with a reduced computational bottleneck. The goal of the direct sampling method is to apply the far-field operator to the far-field pattern of a point source to compute the following two indicator functions:
\begin{align}
    \Ione (z) \coloneqq \big| \langle \phi_z, \mathcal{F} [\phi_z ] \rangle_{L^2(\mathbb{S}^1)} \big|^{\rho/2}, \quad \text{ and } \quad
    \Itwo (z) \coloneqq \| \mathcal{F}[\phi_z] \|^\rho_{L^2(\mathbb{S}^1)},       \label{eq:imaging-functions-DSM2} 
   \end{align}
where $\phi_z(d) = e^{-ik z \cdot d}$ is the far-field pattern of the fundamental solution (up to some rescaling) centered at a point $z$ and $\rho > 0$ is a regularization parameter used to refine the resolution of the imaging functionals.

 While direct sampling approaches have been widely explored for acoustic scattering problems (see, for instance, \cite{li2012direct}), their applicability to biharmonic equations has only recently been explored for the clamped plate problem \cite{HLL-biharmonic}, and our analysis follows closely. The primary theoretical results concern the factorization of the forward operator $\mathcal{G}$, which allows us to predict the decay of the indicator functions above.

\subsection{Integral formulation for the supported plate}\label{sec:supported_IE}

 Adopting the formulation for the supported plate problem appearing in \cite{farkas,nekrasov2025boundary, agocs2026integral}, we write the scattered field $u^s$ in terms of boundary densities $\varphi_1$ and $\varphi_2$: 
\begin{align}
    u^s(x) &= \int_{\partial D} \big( \partial_{n_y}^3 \Phi (x,y) + \alpha_1 \partial_{n_y} \partial_{\tau_y}^2 \Phi (x,y) + \alpha_2 \kappa(y) \partial_{n_y}^2 \Phi (x,y) \nonumber \\
    & \qquad + \alpha_3 \kappa'(y) \partial_{\tau_y} \Phi (x,y) \big) \varphi_1(y)  \, \dd s_y  + \int_{\partial D} \partial_{n_y} \Phi (x,y) \varphi_2(y)  \, \dd s_y , \label{scatteredus}
\end{align}
where $\kappa(y)$ is the curvature at the point $y \in \partial D$, $\kappa'(y)$ is the arc-length derivative of the curvature, and $\partial_{n_y}$ and $\partial_{\tau_y}$ represent the normal and tangential derivatives at $y$, respectively. The coefficients $\alpha_1,\alpha_2,$ and $\alpha_3$ are given by: 
\begin{equation*}
    \alpha_1 = 2- \nu \, , \ \ \alpha_2 = \frac{(-1+\nu)(7+\nu)}{3-\nu} \, , \ \  \alpha_3 = \frac{(1-\nu)(3+\nu)}{1+\nu} \, .
\end{equation*}
Applying the supported plate boundary conditions to $u = u^s + u^i$ leads to the following system of boundary integral equations derived in \cite{nekrasov2025boundary,agocs2026integral}:
\begin{equation}\label{eq:supported_IE}
    \begin{pmatrix}
        -\frac12 I + \mathcal{K}_{11} & \mathcal{K}_{12} \\
        c_0 \kappa^2 I + \mathcal{K}_{21} &  - \frac12 I + \mathcal{K}_{22}
    \end{pmatrix} \begin{pmatrix}
        \varphi_1 \\ \varphi_2
    \end{pmatrix} = - \begin{pmatrix}
        u^i \\ M[u^i] 
    \end{pmatrix},
\end{equation}
where $ c_0 =  \frac{(\nu-1)(\nu+3)(2\nu-1)}{2(3-\nu)} $ and $\mathcal{K}_{11},\mathcal{K}_{12},\mathcal{K}_{21},$ and $\mathcal{K}_{22}$ are boundary integral operators with the following kernels:
\begin{align*}
    K_{11}(x,y) &=  \left[\partial_{n_y}^3 + \alpha_1 \partial_{n_y} \partial_{\tau_y}^2  + \alpha_2 \kappa(y) \partial_{n_y}^2
    + \alpha_3 \kappa'(y) \partial_{\tau_y}\right] \Phi(x,y) \, , \\
    K_{12}(x,y) &= \partial_{n_y} \Phi(x,y) \, , \\
    K_{21}(x,y) &=  \left[\partial_{n_x}^2 \partial_{n_y}^3  + \alpha_1 \partial_{n_x}^2 \partial_{n_y} \partial_{\tau_y}^2 + \alpha_2 \kappa(y) \partial_{n_x}^2 \partial_{n_y}^2  \right.   + \alpha_3 \kappa'(y) \partial_{n_x}^2 \partial_{\tau_y}   \\
    &\quad  +\nu \partial_{\tau_x}^2 \partial_{n_y}^3 +  \nu \alpha_1 \partial_{\tau_x}^2 \partial_{n_y} \partial_{\tau_y}^2  + \nu \alpha_2 \kappa(y) \partial_{\tau_x}^2 \partial_{n_y}^2
    \left.+ \nu \alpha_3 \kappa'(y) \partial_{\tau_x}^2 \partial_{\tau_y} \right]\Phi(x,y) \, , \\
    K_{22}(x,y) &= \left[\partial_{n_x}^2 \partial_{n_y}+ \nu \partial_{\tau_x}^2 \partial_{n_y}\right] \Phi(x,y) \, , 
\end{align*}

Next, we have the following statement about the integral equation above:

\begin{theorem} \label{thm:fredholm}
    Equation \eqref{eq:supported_IE} is Fredholm second kind on $H^{3/2}(\partial D) \times H^{-1/2}(\partial D)$. 
\end{theorem}

Before proving the above theorem, we require the following lemma.

\begin{lemma}
    The operator $\mathcal{K}_{11}$ is compact from $ H^{3/2}(\partial D) \to H^{3/2} (\partial D)$, $\mathcal{K}_{12}$ is compact from $ H^{-1/2}(\partial D) \to H^{3/2} (\partial D)$, $\mathcal{K}_{21}$ is compact from $ H^{3/2}(\partial D) \to H^{-1/2} (\partial D)$, and $\mathcal{K}_{22}$ is compact from $ H^{-1/2}(\partial D) \to H^{-1/2} (\partial D)$.
\end{lemma}

\begin{proof}

Recall that the power series for the fundamental solution $\Phi(x,y)$ centered at $ x = y $ is given by:
\begin{equation}
     \Phi(x ,y ) =   K^S(x ,y ) + \Phi^B(x,y) + \mathcal{O}( |x -y |^6 \ln |x -y | ) ,
\end{equation}
where $K^S(x,y)$ is a $C^\infty$ function and the leading order singularity of $\Phi(x,y)$ is given by the biharmonic Green's function $\Phi^B(x,y) := \frac{1}{8\pi} |x-y|^2 \ln |x-y|$. Let $K_{ij}^B$ denote the kernels above with derivatives applied to $\Phi^B(x,y)$. In \cite{nekrasov2025boundary}, it was shown that $K_{21}^B$ and $K_{22}^B$ are smooth for a smooth curve and, therefore, $\mathcal{K}_{21}^B$ and $\mathcal{K}_{22}^B$ are compact. It was also shown that $ \left[\partial_{n_y}^3 + \alpha_1 \partial_{n_y} \partial_{\tau_y}^2 \right] \Phi^B(x,y) $ is smooth when restricted to the curve, therefore $\mathcal{K}^B_{11} - \overline{\mathcal{K}}^B_{11}$ is compact, where $\overline{\mathcal{K}}^B_{11}$ is the remainder corresponding to the terms $ \left[ \alpha_2 \kappa(y) \partial_{n_y}^2
    + \alpha_3 \kappa'(y) \partial_{\tau_y}\right] \Phi^B(x,y) $. To show compactness of $\overline{\mathcal{K}}^B_{11}$, we note that since $\Phi^B(x,y)$ is the fundamental solution of a fourth-order elliptic PDE, the theory of pseudo-differential operators (e.g. Theorem 9.5.8 in \cite{HsWe2021}) guarantees that $\overline{\mathcal{K}}^B_{11}$ is a continuous operator from $H^{3/2}(\partial D) \to H^{5/2}(\partial D)$ which is compactly embedded in $H^{3/2}(\partial D)$. To show that $\mathcal{K}^B_{12}$ is compact, we note that this kernel can be expanded on the surface as $K^B_{12}(\gamma(s),\gamma(0)) = \frac{1}{8\pi} s^2 \kappa \log(s) + \mathcal{O}(s^3 \log(s))$, where $\gamma(s)$ is the arc-length parametrization of the curve. For a smooth curve, this kernel has the same leading-order singularity as the kernel of the biharmonic single layer potential $\int_{\partial D} G^B(x,y) \varphi(y) \, \dd s_y$, which by Theorem 9.5.8 in \cite{HsWe2021} is a continuous operator from $H^{-1/2}(\partial D) \to H^{5/2}(\partial D)$. Therefore, $\mathcal{K}_{12}^B$ is compact from $H^{-1/2}(\partial D) \to H^{3/2}(\partial D)$. The next singular term in the power series for $\Phi(x,y)$ is $|x-y|^6 \log(|x-y|)$, which is the fundamental solution to an eighth-order elliptic PDE. Since we are taking at most five derivatives in the kernels above, the rest of the terms in the power series correspond to compact operators. 
\end{proof}

Theorem \ref{thm:fredholm} follows immediately from the lemma above. For convenience, let $\mathcal{Q}$ denote the operator appearing on the left-hand side of \eqref{eq:supported_IE}:
\begin{equation*}
    \mathcal{Q} := 
    \begin{pmatrix}
        -\frac12 I + \mathcal{K}_{11} & \mathcal{K}_{12} \\
        c_0 \kappa^2 I + \mathcal{K}_{21} &  - \frac12 I + \mathcal{K}_{22}
    \end{pmatrix}.
\end{equation*}
By analytic Fredholm theory, the operator $\mathcal{Q}$ is  invertible except at a discrete set of wavenumbers. Applying the bounded inverse theorem, there exists a $\mathcal{Q}^{-1}: H^{3/2}(\partial D) \times H^{-1/2}(\partial D) \to H^{3/2}(\partial D) \times H^{-1/2} (\partial D)$ for almost all wavenumbers such that:
\begin{equation*}
    \begin{pmatrix}
        \varphi_1 \\[2pt]
        \varphi_2
    \end{pmatrix}
    = - \mathcal{Q}^{-1}
    \begin{pmatrix}
        u^i \\[2pt]
M[u_i]   \end{pmatrix}.
\end{equation*}
Then, from \eqref{eq:Phi_ff} and \eqref{scatteredus}, the far-field pattern of the scattered field is given by:
\begin{align}
    u^\infty(\hat{x}) &= \int_{\partial D} \big( \partial_{n_y}^3 e^{-ik\hat{x}\cdot y} + \alpha_1 \partial_{n_y} \partial_{\tau_y}^2 e^{-ik\hat{x}\cdot y} + \alpha_2 \kappa(y) \partial_{n_y}^2 e^{-ik\hat{x}\cdot y} \nonumber \\
    & \qquad + \alpha_3 \kappa'(y) \partial_{\tau_y} e^{-ik\hat{x}\cdot y} \big) \varphi_1(y)  \, \dd s_y  + \int_{\partial D} \partial_{n_y} e^{-ik\hat{x}\cdot y} \varphi_2(y)  \, \dd s_y , 
\end{align}
Therefore, we define the following operator $\mathcal{M}^{\infty} : H^{3/2} (\partial D) \times H^{-1/2}(\partial D) \to L^2(\mathbb{S}^1)$ which maps from boundary densities to the far-field pattern:
\begin{align*}
    \mathcal{M}^{\infty}   [(
        \varphi_1, 
        \varphi_2 )^T](\hat{x}) &:= \int_{\partial D} \big( \partial_{n_y}^3 e^{-ik\hat{x}\cdot y} + \alpha_1 \partial_{n_y} \partial_{\tau_y}^2 e^{-ik\hat{x}\cdot y} + \alpha_2 \kappa(y) \partial_{n_y}^2 e^{-ik\hat{x}\cdot y} \nonumber \\
    & \qquad + \alpha_3 \kappa'(y) \partial_{\tau_y} e^{-ik\hat{x}\cdot y} \big) \varphi_1(y)  \, \dd s_ y  + \int_{\partial D} \partial_{n_y} e^{-ik\hat{x}\cdot y} \varphi_2(y)  \, \dd s_y , 
\end{align*}
Consequently, the far-field pattern for a supported cavity $D$ for an incident plane wave $u^i$ can be represented as
\begin{equation}
    u^{\infty}(\hat{x}, d) = \mathcal{M}^{\infty} \mathcal{Q}^{-1}
    \begin{pmatrix}
        u^i \\[2pt]
M[u^i]    \end{pmatrix} ,
    \label{farfield-expression}
\end{equation}
allowing us to write the far‑field operator as 
\begin{equation}\mathcal{F} = \mathcal{M}^{\infty} \mathcal{Q}^{-1} \mathcal{H} \, . \end{equation} 
This factorization will be useful for proving properties of the indicator functions used for the DSM in the following section.

\subsection{Properties of the imaging functions}

The next result provides an estimate on the rate of decay for the indicator function $\Ione$ away from the obstacle $D$.

\begin{theorem}\label{besselupperbound}
Let $\Ione$ be defined as above with regularization parameter $\rho > 0$. Then, for almost all wavenumbers, $\Ione$ satisfies:
\begin{equation*}
    \big| \mathcal{I}_{\mathrm{DSM\text{-}1}}(x)| = \mathcal{O} ( \dis (x,D)^{-\rho/2} ) , \quad \mathrm{ as } \quad  \mathrm{dist}(x, D) \to \infty,
\end{equation*}
for some compact domain $D$, where $\dis(x,D) := \inf \{ | x- y | \; | \; y \in D \}$.
\end{theorem}
\begin{proof}
Given the factorization $\mathcal{F} =  \mathcal{M}^{\infty} \mathcal{Q}^{-1} \mathcal{H}$, the operator $\mathcal{Q}^{-1} $ is bounded for almost all wavenumbers. Therefore, we can use Cauchy-Schwarz to write:
\begin{align}
    \big| \langle \phi_z, \mathcal{F} [\phi_z] \rangle_{L^2(\mathbb{S}^1)} \big|^{\rho/2}
    &= \big| \langle (\mathcal{M}^{\infty})^{*} [\phi_z], \mathcal{Q}^{-1} \mathcal{H} [\phi_z] \rangle \big|^{\rho/2} \nonumber \\
    &\leq C \| (\mathcal{M}^{\infty})^{*} [\phi_z] \|_{H^{-3/2} (\partial D) \times H^{1/2}(\partial D) }^{\rho/2} \| \mathcal{H} [\phi_z] \|_{H^{3/2} (\partial D) \times H^{-1/2}(\partial D) }^{\rho/2}  \label{eq:innerprod}
\end{align}
where $(\mathcal{M}^{\infty})^{*}$ is given by
\begin{align*}
( \mathcal{M}^{\infty} )^* [\phi_z] (x) &= \begin{bmatrix} \partial_{n_x}^3  + \alpha_1 \partial_{n_x} \partial_{\tau_x}^2 + \alpha_2 \kappa(x) \partial_{n_x}^2   + \alpha_3 \kappa'(x) \partial_{\tau_x}   \\ 
\partial_{n_x} \end{bmatrix} \int_{\mathbb{S}^1}   e^{ikx \cdot d} \phi_z(d) \, \dd s(d)  \\
&= 2\pi \begin{bmatrix} \partial_{n_x}^3  + \alpha_1 \partial_{n_x} \partial_{\tau_x}^2 + \alpha_2 \kappa(x) \partial_{n_x}^2   + \alpha_3 \kappa'(x) \partial_{\tau_x}   \\ 
\partial_{n_x} \end{bmatrix}  J_0(k|x-z|)
\end{align*}
for $x \in \partial D$ 
where $J_0$ denotes the Bessel function of the first kind of order zero. Then, we can use the trace theorem to write:
\begin{equation*}
    \| (\mathcal{M}^{\infty})^{*} [\phi_z] \|_{H^{-3/2} (\partial D) \times H^{1/2}(\partial D) } \leq C \| J_0(k| \,  \cdot \, -z|) \|_{H^2( D )}
\end{equation*}
The other term in inequality \eqref{eq:innerprod} can be written in terms of $J_0$ as well, i.e. 
\begin{equation*}
    \mathcal{H}[\phi_z](x) = 2\pi \big( J_0(k|x - z|), \,  M[J_0(k|\,\cdot \,-z|)](x) \, \big)^T ,
\end{equation*}    
Again, we use the trace theorem to say that 
\begin{equation*}
     \| \mathcal{H} [\phi_z] \|_{H^{3/2} (\partial D) \times H^{-1/2}(\partial D) } \leq C \| J_0(k |\cdot - z | ) \|_{H^2(D)} 
\end{equation*}

Finally, in Lemma 3.2 of  \cite{HLL-biharmonic}, the authors showed that for $|\alpha|\leq 4$ $ \| D^\alpha J_0(|x- \, \cdot \, |) \|_{L^2(D)} = \mathcal{O}(\dis (x,D)^{-1} ) $ as $\dis (x,D) \to \infty$, from which the result follows immediately.
\end{proof}

Now, we turn to the second imaging functional of interest, defined by
\begin{equation*}
\Itwo(x) = \| \mathcal{F} [\phi_z ] \|_{L^2(\mathbb{S}^1)}^{\rho}, \qquad \rho > 0.
\end{equation*}

Though the behavior of this second functional is similar to that of the first, we include it here for completeness. It is possible to show that this functional has the same decay properties as the first functional, and indeed, this was recently shown for the clamped plate problem. We state this result and refer the reader to \cite{HLL-biharmonic} for the proof:

\begin{theorem}\label{resolution2}
Let $\Itwo$ be defined as above with regularization parameter $\rho > 0$. Then, for almost all wavenumbers, $\Itwo$ satisfies:
\begin{equation*}
    \big| \Itwo (x)| = \mathcal{O} ( \dis (x,D)^{-\rho/2} ) , \quad \mathrm{ as } \quad  \mathrm{dist}(x, D) \to \infty,
\end{equation*}
for some compact domain $D$, where $\dis(x,D) := \inf \{ | x- y | \; | \; y \in D \}$.
\end{theorem}

Using \Cref{resolution2}, it is evident that the two imaging functionals under consideration are equivalent up to multiplicative constants. Specifically, there exist positive constants \( c_1, c_2 > 0 \) such that
\begin{equation*}
c_1 | \Ione (x) |
\le 
 | \Itwo(x)  | \le  
c_2  | \Ione(x) | .
    \end{equation*}
for all $x \in \mathbb{R}^2 \setminus \overline{D}$. As a result, both imaging functionals are expected to provide comparable reconstructions of the obstacle $D$, and indeed, this is what is observed numerically.

\section{Numerical experiments} \label{sec:num_results}
In this section, we first discuss the numerical implementation of the methods before presenting numerical experiments for the LSM and DSM. In the first experiment, we use simulated noise-free data at different frequencies to assess the qualitative performance of the two methods. In the second experiment, we repeat the same setup but introduce additive and multiplicative noise of varying intensities. In the third experiment, we investigate the behavior of the methods for different values of the Poisson's ratio. In the fourth experiment, we consider the reconstruction of multiple obstacles. Finally, we study the effect of limited available data on the reconstruction of single and multiple obstacles. 

\subsection{Numerical implementation}
We now discuss the implementation of the methods. For both approaches, we must discretize the far‑field operator. Since $d=(\cos(\theta_d),\sin(\theta_d) )\in \mathbb{S}^1$, we discretize the interval $[0,2\pi)$ uniformly to obtain $d_j=(\cos(2j\pi/N_d),\sin(2j\pi/N_d))$, $j=1,\ldots,N_d$, where $N_d$ denotes the number of incident directions used to probe the domain. Using these points, we approximate the integral in the far‑field operator via the trapezoidal rule:
\begin{equation}
    \mathcal{F}[g_z](\hat{x})\approx \frac{2\pi}{N_d}\sum_{j=1}^{N_d} u^{\infty}(\hat{x},d_j)g_z(d_j).
\end{equation}

For the LSM, we solve the regularized minimization problem \eqref{eq:farfieldeq_regularized}, which is equivalent to solving
\begin{equation}
    (\alpha\mathcal{I}+\mathcal{F}^*\mathcal{F}) [ g_z ](\hat{x}) = \mathcal{F}^* [\phi_z ] (\hat{x}),
\end{equation}
where $\mathcal{I}$ is the identity operator. Taking $\hat{x}=(\cos(\theta_r),\sin(\theta_r))\in \mathbb{S}^1$, we can discretize the interval $[0,2\pi)$ uniformly, to obtain $\hat{x}_\ell= ( \cos( 2 \ell \pi/N_r ) , \sin ( 2 \ell \pi / N_r) )$, $\ell =1,\ldots,N_r$, where $N_r$ is the number of receivers used to measure the scattered data. 

With these discretizations, we obtain the linear system
\begin{equation} \label{eq:discrete_LSM}
 (\alpha {\bf I+F^*F}){\bf g}_z={\bf F^*}{\boldsymbol \phi}_z,
\end{equation}
where ${\bf F}$ is an $N_r\times N_d$ matrix with elements ${\bf F}_{\ell j}=(2\pi/N_d)u^{\infty}(\hat{x}_\ell ,d_j)$, $\mathbf{F}^*$ is its adjoint, ${\bf I}$ is the $N_d\times N_d$ identity matrix, ${\bf g}_z$ is a vector with $N_d$ components $g_z(d_j)$ being our variables, and ${\boldsymbol \phi}_z$ is a vector with $N_r$ components $\phi_z(\hat{x}_\ell)$. After computing $g_z$ at all sampling points $z$, we evaluate the LSM indicator function \eqref{eq:lsm-indicator}.

Regarding regularization, in the LSM we tested $\alpha = 10^{-j}$, $j = 1,\ldots,6$, and in each experiment we present the reconstruction giving the best qualitative result; reconstructions for other values of $\alpha$ are very similar.  
For the DSM, since the scaling parameter $\rho > 0$ affects only visualization and does not influence the localization properties of the indicators, we set $\rho = 2$ for $\Ione$ and $\rho = 1$ for $\Itwo$; other choices lead to visually comparable results. Finally, in all figures, the graphs are rescaled so that the maximum value of the indicator function is one. We display the rescaled DSM indicators directly, whereas for the LSM we plot the rescaled value of $\log(\mathcal{I}_{\mathrm{LSM}})$ to enhance contrast and improve visualization.

To obtain the far-field measurements needed in the discretization of the far-field operator $\mathcal{F}$, we use the boundary integral formulation described in \Cref{sec:supported_IE} and originally appearing in \cite{farkas,nekrasov2025boundary}. The forward solver is implemented using the MATLAB software package \texttt{chunkIE} \cite{chunkIE}. In this package, the boundary of the domain is divided into panels, and the integral operators are discretized using generalized Gaussian quadratures. The integral equation \eqref{eq:supported_IE} is then solved using a Nystr\"om method via GMRES. The calculation of the far-field pattern is done using the native smooth quadrature.

Because equation \eqref{eq:discrete_LSM} must be solved for every sampling point $z$, it is efficient to precompute a factorization of the left-hand side once (e.g., using an SVD of $\mathbf{F}$) and reuse it for each right-hand side. Computing such a factorization costs $\mathcal{O}\big(\min(N_d N_r^2,\; N_r N_d^2)\big)$, and applying it to each right-hand side costs $\mathcal{O}(N_d^2)$. Thus, for a total of $N_g$ sampling points, the overall complexity of LSM is $\mathcal{O}\big(\min(N_d N_r^2,\; N_r N_d^2)\big)+\mathcal{O}(N_g N_d^2)$. If $N_d = N_r$, this simplifies to
$\mathcal{O}(N_d^3) + \mathcal{O}(N_g N_d^2)$. In contrast, the DSM implementation is much simpler. No ill-posed equation must be solved, and therefore no regularization is needed. The indicator functions \eqref{eq:imaging-functions-DSM2} are simply evaluated at each sampling point. The dominant computational cost is computing the product $\mathbf{F}\boldsymbol{\phi}_z$, where $(\boldsymbol{\phi}_z)_j = \phi_z(d_j)$. In general, the cost of the DSM is $\mathcal{O}(N_g N_r N_d )$. When $\mathbf{F}$ is square with $N_d = N_r$, the overall DSM cost is $\mathcal{O}(N_g N_d^2)$.

We use three obstacle shapes in our experiments: a 5-arms star-shaped domain parameterized by
\begin{equation*}
    \gamma_A(t) = (1+0.3\cos(5t))(\cos(t),\sin(t)), \quad t\in[0,2\pi),
\end{equation*}
an 11-arms star-shaped domain parameterized by
\begin{equation*}
    \gamma_B(t) = (1+0.5\cos(11t))(\cos(t),\sin(t)), \quad t\in[0,2\pi),
\end{equation*}
and a domain with a pronounced cavity. We will refer to those domains as 5-arms, 11-arms, and cavity throughout the rest of this text. 

\begin{figure}[!ht]
    \centering
    \begin{subfigure}[b]{0.3\linewidth}
      \centering
      \includegraphics[width=1\linewidth]{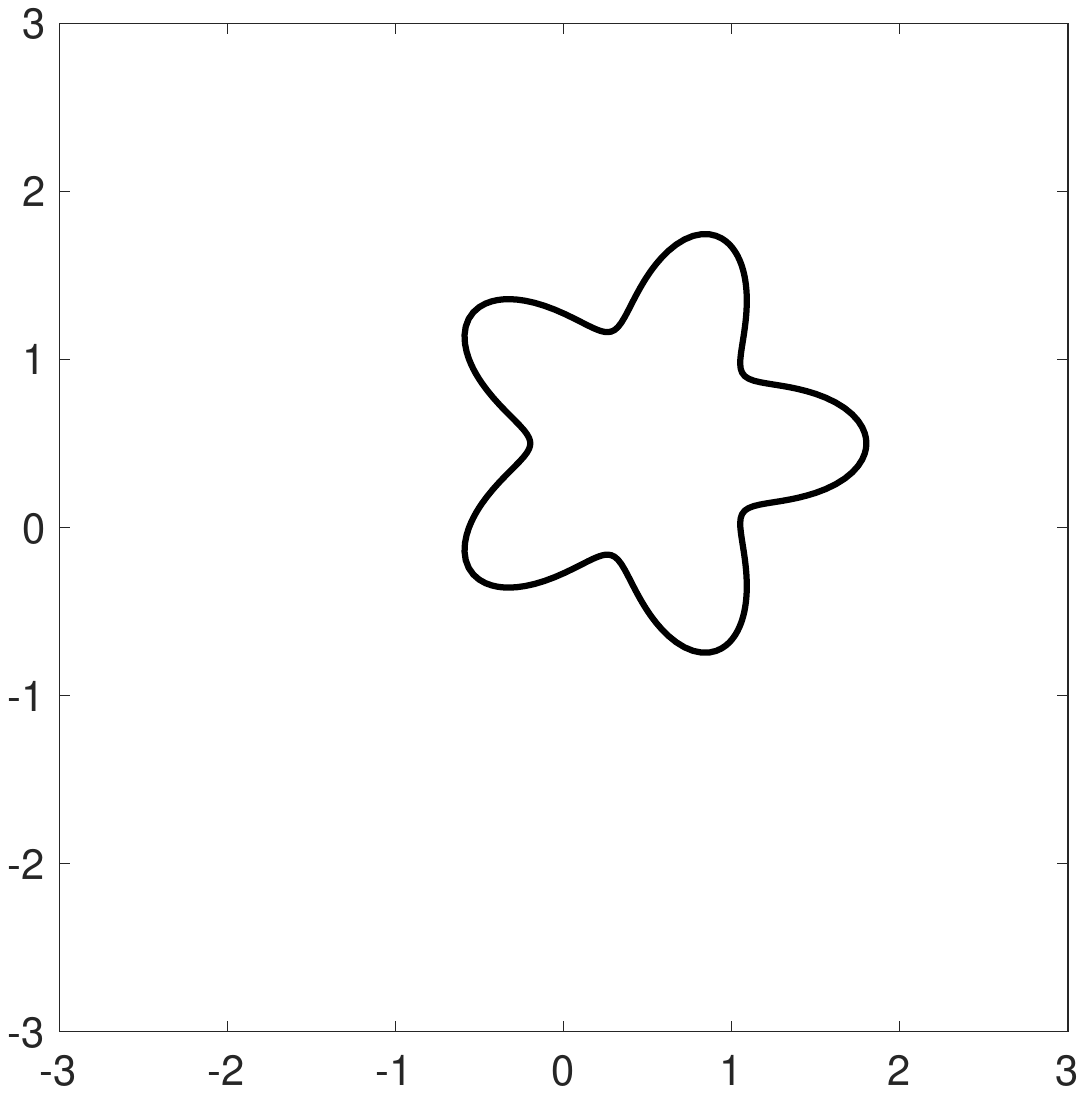}
      \caption{5-arms}
      \label{fig:5arms}
    \end{subfigure} \quad
    \begin{subfigure}[b]{0.3\linewidth}
      \centering
      \includegraphics[width=1\linewidth]{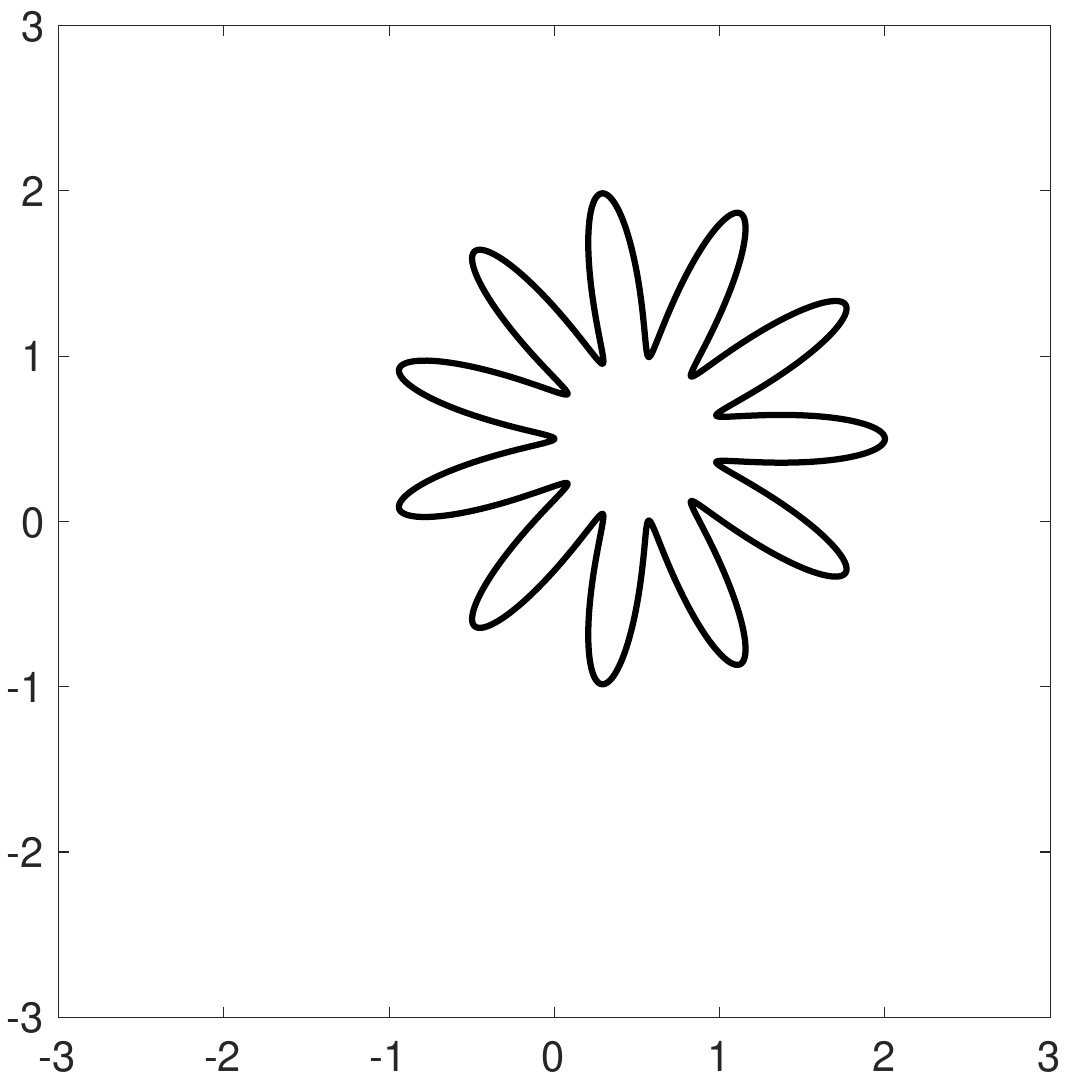}
      \caption{11-arms}
      \label{fig:11arms}
    \end{subfigure}   \quad
    \begin{subfigure}[b]{0.3\linewidth}
      \centering
      \includegraphics[width=1\linewidth]{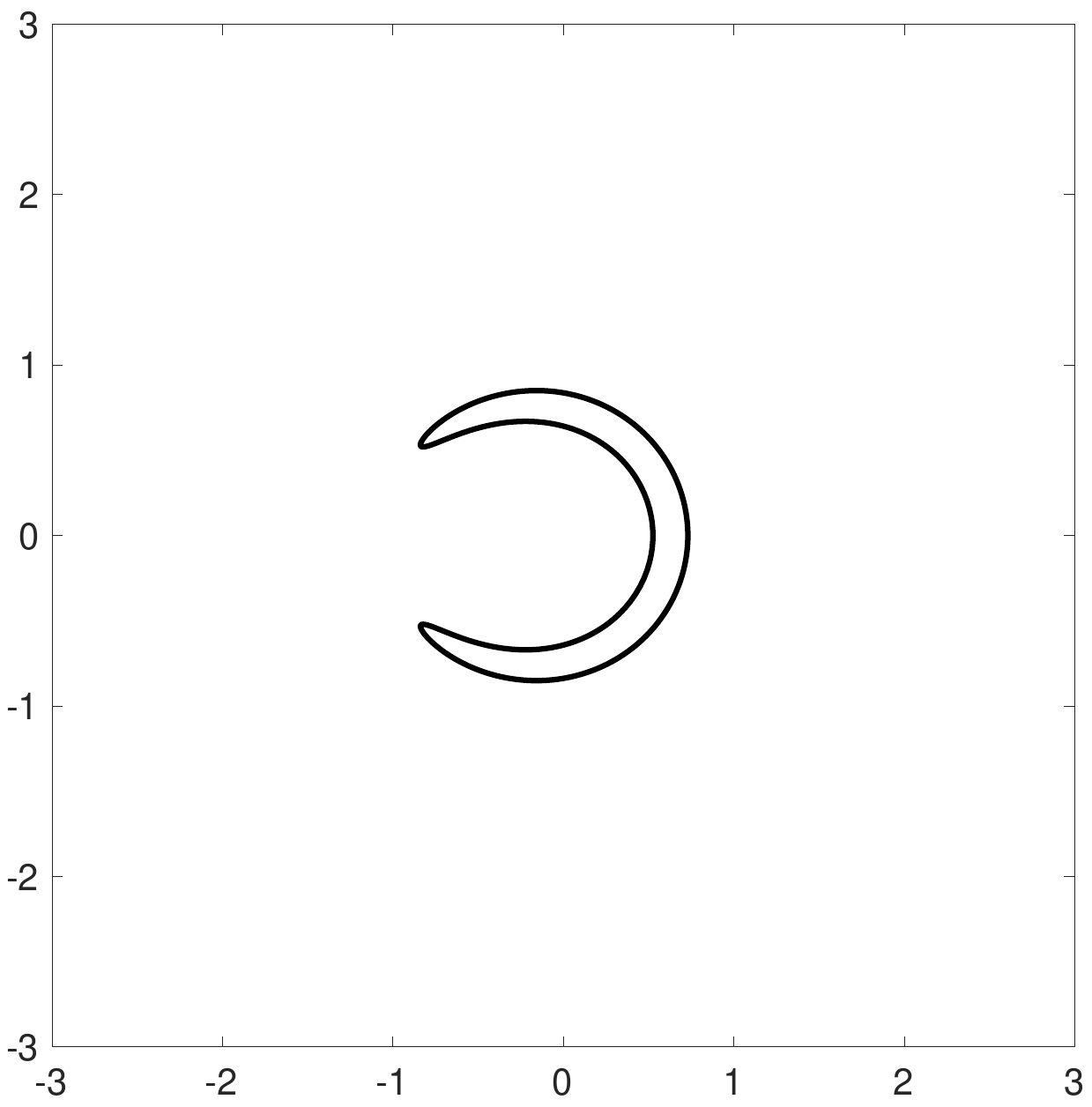}
      \caption{cavity}
      \label{fig:plane}
    \end{subfigure}        
\caption{Ground-truth domains used in the simulations: a) 5-arms, b) 11-arms, and c) cavity.}    
    \label{fig:domains}
\end{figure}

In all experiments -- except for Example 3, where we analyze the effect of the Poisson's ratio -- we use $\nu=0.3$, which is typical for materials such as ice and steel. In Example 3, we consider $\nu=0.5$ (close to rubber), $0$ (cork), and $-0.5$ (auxetic metamaterials).


We test the numerical accuracy of the forward solver for each of the wavenumbers and each of the domains and report it in Table \ref{tab:disc}. The table lists the number of points $N$ used on the boundary for the discretization and the error in the calculation of the solution. The error was assessed using an analytic solution test by placing a point source inside the obstacle and solving for the field at a collection of 10 arbitrary points outside. Then, the $\ell_\infty$ norm of the error at those points was calculated relative to the $L^1$ norm of the solution to the integral equation. Note that the forward solver provides at least 7 digits of accuracy for all the discretizations used.

\begin{table}[h]
    \centering    
    \caption{Discretization details and accuracy of the computed field.}
    \label{tab:disc}
    \begin{tabular}{|c|c|c|c|}
        \hline
        Domain & $k$ & $N$ & Relative error \\
        \hline \hline
        
        \multirow{2}{*}{5-arms} 
            & $2\pi$  & 288  & $1.26\times10^{-7}$ \\ \cline{2-4}
            & $20\pi$ &  2976    &  $2.46\times 10^{-12} $                  \\ \hline
        
        \multirow{2}{*}{11-arms} 
            & $2\pi$  & 3776 & $2.42\times10^{-7}$ \\ \cline{2-4}
            & $20\pi$ &   9344   &  $2.43\times 10^{-10}$                   \\ \hline
        
        \multirow{2}{*}{Cavity} 
            & $2\pi$  & 768  & $1.63\times10^{-8}$ \\ \cline{2-4}
            & $20\pi$ &   3584   & $3.84\times 10^{-14}$                     \\ \hline
        
        \multirow{2}{*}{Multiple} 
            & $2\pi$  &   1248   &  $2.22\times 10^{-9}$                    \\ \cline{2-4}
            & $10\pi$ &   5376   &  $3.00\times 10^{-14}$                   \\ \hline
        
    \end{tabular} 
        
\end{table}

\subsection{Example 1 -- Resolution}
The main objective of this example is to illustrate the performance of the method at different frequencies. In particular, we chose $k=2\pi$ and $20 \pi$ so that the wavelength of the incoming wave is $\lambda=1$ and $0.1$, respectively. The indicator function is calculated on a uniform grid over the domain $[-3,3]$ with $300$ points in each direction, totaling $N_g=300^2$. We use $N_d=N_r=128$ for $k=2\pi$ and $N_d=N_r=1024$ for $k=20\pi$. For the LSM, we set the regularization parameter to $\alpha = 10^{-4}$ based on empirical evidence that this choice of regularization parameter best matches the original shape.

The reconstructions for $k=2\pi$ are shown in Figure~\ref{fig:ex1_rec2pi}, and those for $k=20\pi$ in Figure~\ref{fig:ex1_rec20pi}. In all cases, the general location and what seems to be a substantial portion of the convex hull of the obstacles are recovered. For the 5-arms domain, the reconstructions are very close to the ground truth; however, for the 11-arms domain, the methods fail to resolve the small cavities along the boundary. A similar effect occurs in the reconstruction of the cavity-shaped domain, where most of the cavity structure is lost.

At higher frequency, both methods yield more detailed reconstructions. Overall, both methods are effective for obtaining a coarse approximation of the domain, but they are unable to recover fine-scale features, particularly cavities. When extracting a level-set curve to represent the reconstructed obstacle, one should therefore expect a good approximation of its convex hull rather than its detailed boundary structure.

The behavior observed in this example mirrors observations in related Helmholtz scattering problems; for example, \cite{askham2024random} shows that cavities are difficult to recover even with iterative methods in the Helmholtz setting. In \cite{borges2023robustness}, some cavities were reconstructed for penetrable obstacles, and we expect similar results could be achieved in the biharmonic case for penetrable inclusions as well, though this remains unclear for impenetrable ones.

\begin{figure}[!ht]
    \centering
    \begin{subfigure}[b]{0.3\linewidth}
      \centering
      \includegraphics[width=1\linewidth]{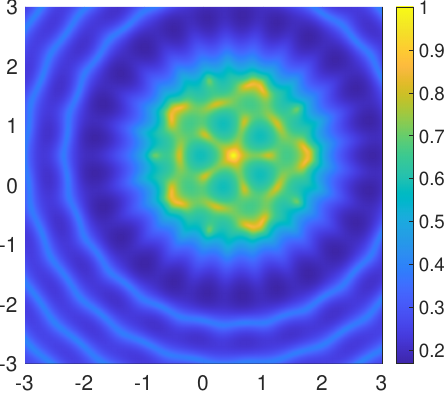}
      \caption{LSM, $k=2\pi$, 5-arms}
      \label{fig:5arms_k2pi_lsm}
    \end{subfigure} \quad
    \begin{subfigure}[b]{0.3\linewidth}
      \centering
      \includegraphics[width=1\linewidth]{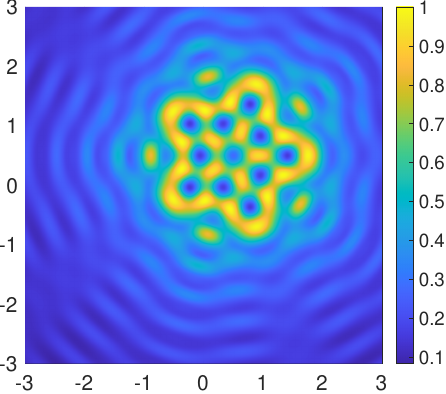}
      \caption{DSM--1, $k=2\pi$, 5-arms}
      \label{fig:5arms_k2pi_dsm1}
    \end{subfigure}   \quad
    \begin{subfigure}[b]{0.3\linewidth}
      \centering
      \includegraphics[width=1\linewidth]{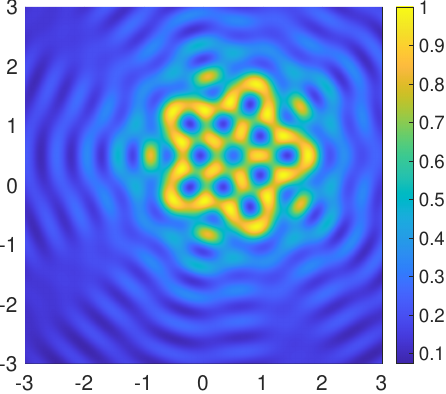}
      \caption{DSM--2, $k=2\pi$, 5-arms}
      \label{fig:5arms_k2pi_dsm2}
    \end{subfigure}        

    \centering
    \begin{subfigure}[b]{0.3\linewidth}
      \centering
      \includegraphics[width=1\linewidth]{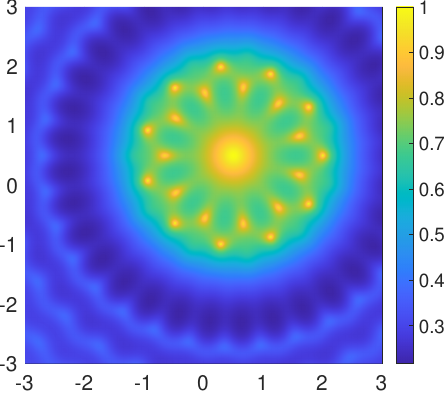}
      \caption{LSM, $k=2\pi$, 11-arms}
      \label{fig:11arms_k2pi_lsm}
    \end{subfigure} \quad
    \begin{subfigure}[b]{0.3\linewidth}
      \centering
      \includegraphics[width=1\linewidth]{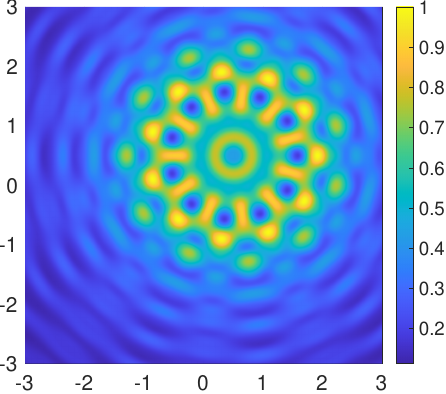}
      \caption{DSM--1, $k=2\pi$, 11-arms}
      \label{fig:11arms_k2pi_dsm1}
    \end{subfigure}   \quad
    \begin{subfigure}[b]{0.3\linewidth}
      \centering
      \includegraphics[width=1\linewidth]{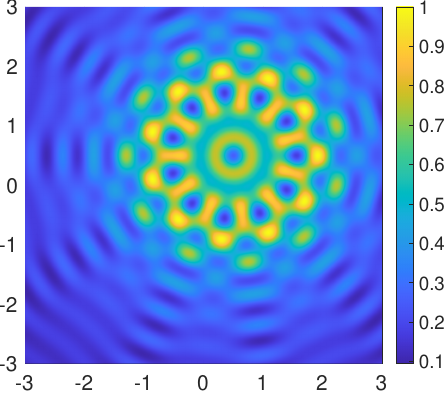}
      \caption{DSM--2, $k=2\pi$, 11-arms}
      \label{fig:11arms_k2pi_dsm2}
    \end{subfigure}  

    \centering
    \begin{subfigure}[b]{0.3\linewidth}
      \centering
      \includegraphics[width=1\linewidth]{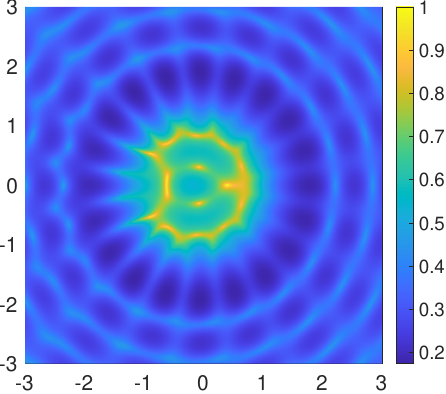}
      \caption{LSM, $k=2\pi$, cavity}
      \label{fig:cavity_k2pi_lsm}
    \end{subfigure} \quad
    \begin{subfigure}[b]{0.3\linewidth}
      \centering
      \includegraphics[width=1\linewidth]{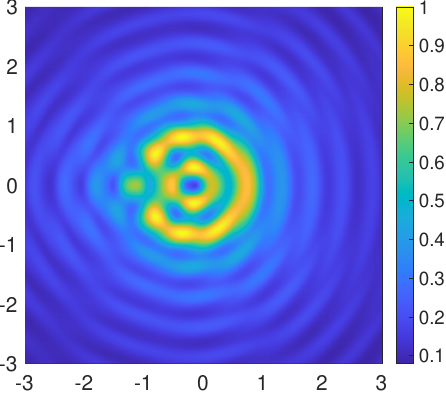}
      \caption{DSM--1, $k=2\pi$, cavity}
      \label{fig:cavity_k2pi_dsm1}
    \end{subfigure}   \quad
    \begin{subfigure}[b]{0.3\linewidth}
      \centering
      \includegraphics[width=1\linewidth]{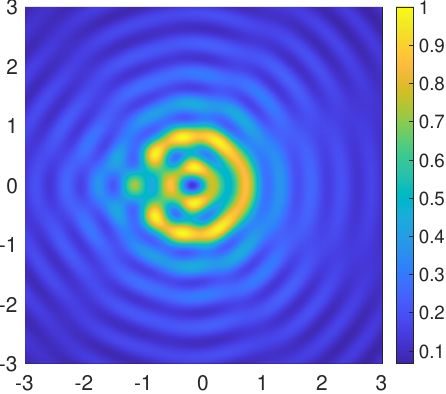}
      \caption{DSM--2, $k=2\pi$, cavity}
      \label{fig:cavity_k2pi_dsm2}
    \end{subfigure}  

\caption{Reconstructions using the LSM (left column), DSM--1 (middle column), and DSM--2 (right column) at wavenumber $k=2\pi$: the 5-arms domain (top row), the 11-arms domain (middle row), and the cavity (bottom row).}    
    \label{fig:ex1_rec2pi}
\end{figure}

\begin{figure}[!ht]
    \centering
    \begin{subfigure}[b]{0.3\linewidth}
      \centering
      \includegraphics[width=1\linewidth]{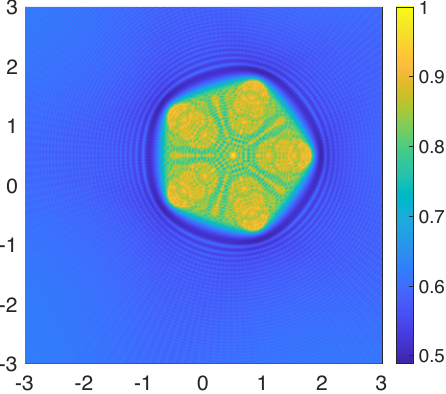}
      \caption{LSM, $k=20\pi$, 5-arms}
      \label{fig:5arms_k20pi_lsm}
    \end{subfigure} \quad
    \begin{subfigure}[b]{0.3\linewidth}
      \centering
      \includegraphics[width=1\linewidth]{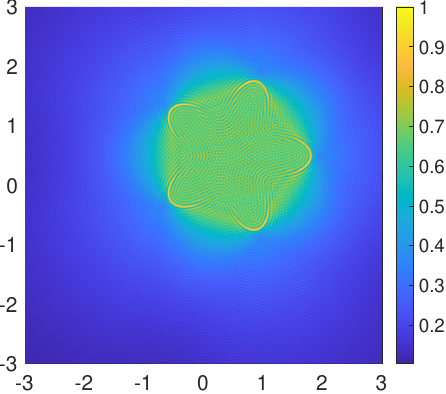}
      \caption{DSM--1, $k=20\pi$, 5-arms}
      \label{fig:5arms_k20pi_dsm1}
    \end{subfigure}   \quad
    \begin{subfigure}[b]{0.3\linewidth}
      \centering
      \includegraphics[width=1\linewidth]{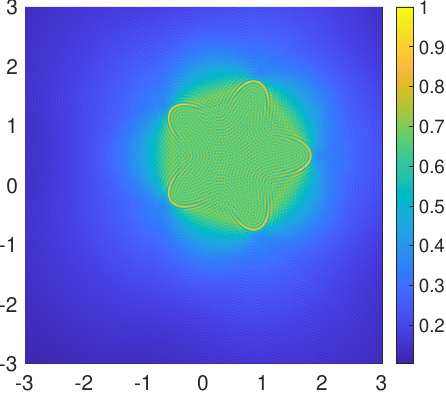}
      \caption{DSM--2, $k=20\pi$, 5-arms}
      \label{fig:5arms_k20pi_dsm2}
    \end{subfigure}        

    \centering
    \begin{subfigure}[b]{0.3\linewidth}
      \centering
      \includegraphics[width=1\linewidth]{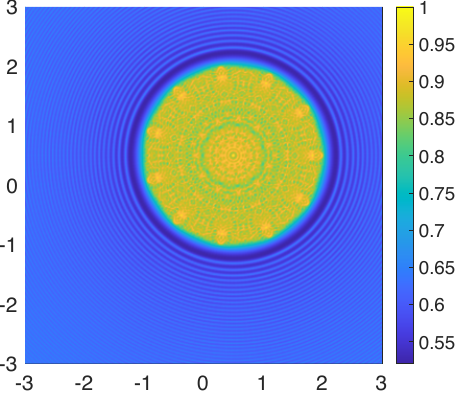}
      \caption{LSM, $k=20\pi$, 11-arms}
      \label{fig:11arms_k20pi_lsm}
    \end{subfigure} \quad
    \begin{subfigure}[b]{0.3\linewidth}
      \centering
      \includegraphics[width=1\linewidth]{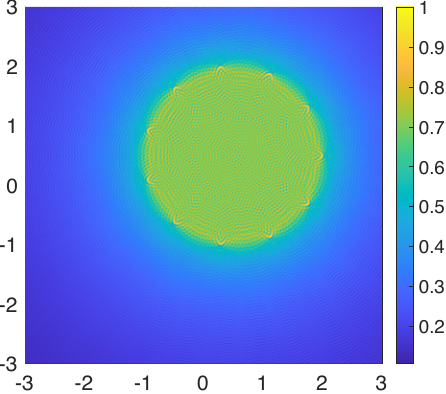}
      \caption{DSM--1, $k=20\pi$, 11-arms}
      \label{fig:11arms_k20pi_dsm1}
    \end{subfigure}   \quad
    \begin{subfigure}[b]{0.3\linewidth}
      \centering
      \includegraphics[width=1\linewidth]{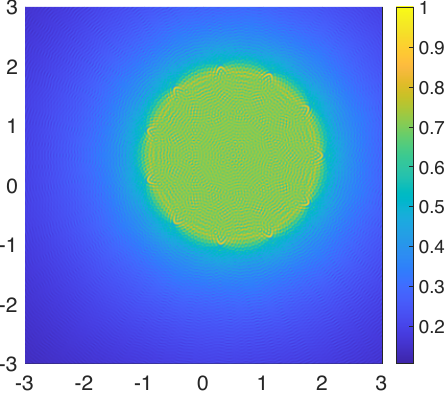}
      \caption{DSM--2, $k=20\pi$, 11-arms}
      \label{fig:11arms_k20pi_dsm2}
    \end{subfigure}  

    \centering
    \begin{subfigure}[b]{0.3\linewidth}
      \centering
      \includegraphics[width=1\linewidth]{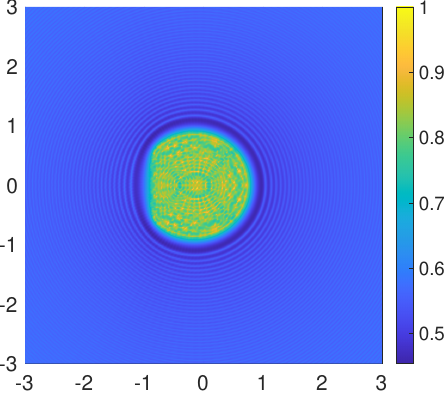}
      \caption{LSM, $k=20\pi$, cavity}
      \label{fig:cavity_k20pi_lsm}
    \end{subfigure} \quad
    \begin{subfigure}[b]{0.3\linewidth}
      \centering
      \includegraphics[width=1\linewidth]{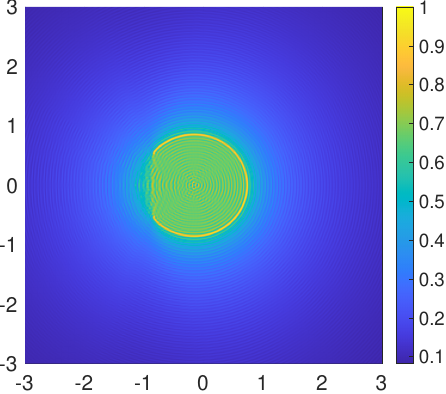}
      \caption{DSM--1, $k=20\pi$, cavity}
      \label{fig:cavity_k20pi_dsm1}
    \end{subfigure}   \quad
    \begin{subfigure}[b]{0.3\linewidth}
      \centering
      \includegraphics[width=1\linewidth]{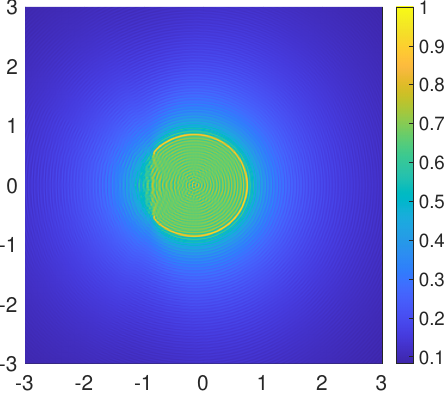}
      \caption{DSM--2, $k=20\pi$, cavity}
      \label{fig:cavity_k20pi_dsm2}
    \end{subfigure}  

\caption{Reconstructions using the LSM (left column), DSM--1 (middle column), and DSM--2 (right column) at wavenumber $k=20\pi$ for the 5-arms domain (top row), the 11-arms domain (middle row), and the cavity (bottom row).}    
    \label{fig:ex1_rec20pi}
\end{figure}

\subsection{Example 2 -- Noisy data}
Next, we examine the behavior of the methods when the available data contain two types of noise: additive noise and multiplicative noise. We use the 5-arms domain for this experiment, and the data are generated at the wavenumber $k=2\pi$. The number of incident directions and receivers is given by $N_d=N_r=128$.

To add noise, we generate a complex random number $\delta\in \mathbb{C}$ whose real and imaginary parts are independently sampled from a normal distribution with zero mean and unit variance. The measurements used in this experiment $\tilde{u}^\infty$ are obtained by adding noise to the noiseless data $u^\infty$. For the additive noise, the noise is added using the formula:
\begin{equation*}
    \tilde{u}^\infty = u^\infty + c_i \frac{\delta}{|\delta|} |u^\infty|,
\end{equation*}
and for the multiplicative noise, the noise is added using the formula:
\begin{equation*}
    \tilde{u}^\infty = u^\infty + c_i \frac{\delta}{|\delta|},
\end{equation*}
where $c_i$ is the noise level. We choose $c_i=0.05$, $0.5$, and $1$, corresponding to $5\%$, $50\%$, and $100\%$, respectively.

For the LSM reconstruction, the regularization parameter was chosen empirically as $\alpha = 10^{-1}$ for all noise levels and noise types. The indicator function for all methods is computed on the same grid used in the previous example.

The results are presented in Figures \ref{fig:ex2_rec2pi_add} and \ref{fig:ex2_rec2pi_multi}, corresponding to the additive and multiplicative noisy data, respectively. Note that the DSM is more robust to noise than the LSM, as it produces reconstructions that remain much closer to the original domain even as the noise level increases. In contrast, the LSM reconstruction deteriorates under both noise types, with an increased degradation in the presence of additive noise.

\begin{figure}[!ht]
    \centering
    \begin{subfigure}[b]{0.3\linewidth}
      \centering
      \includegraphics[width=1\linewidth]{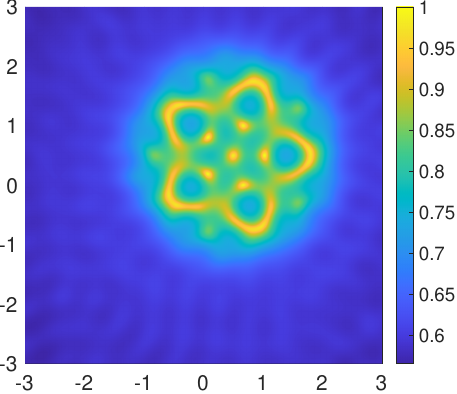}
      \caption{LSM, additive, $\delta=0.05$}
      \label{fig:5arms_k2pi_lsm_add005}
    \end{subfigure} \quad
    \begin{subfigure}[b]{0.3\linewidth}
      \centering
      \includegraphics[width=1\linewidth]{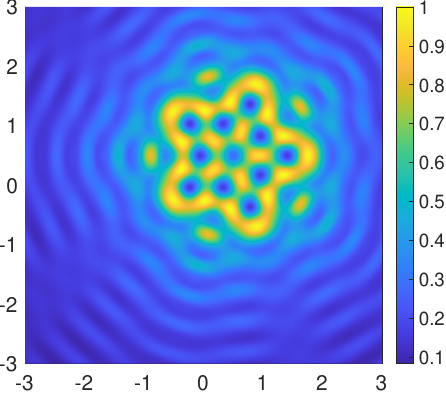}
      \caption{DSM--1, additive, $\delta=0.05$}
      \label{fig:5arms_k2pi_dsm1_add005}
    \end{subfigure}   \quad
    \begin{subfigure}[b]{0.3\linewidth}
      \centering
      \includegraphics[width=1\linewidth]{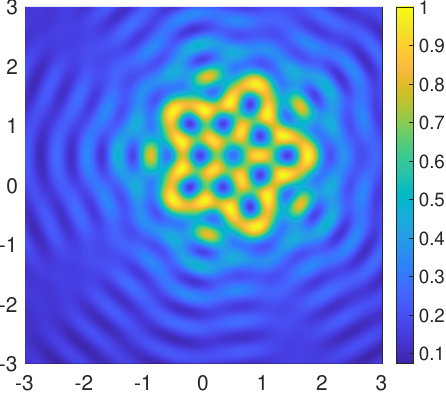}
      \caption{DSM--2, additive, $\delta=0.05$}
      \label{fig:5arms_k2pi_dsm2_add005}
    \end{subfigure}        

    \centering
    \begin{subfigure}[b]{0.3\linewidth}
      \centering
      \includegraphics[width=1\linewidth]{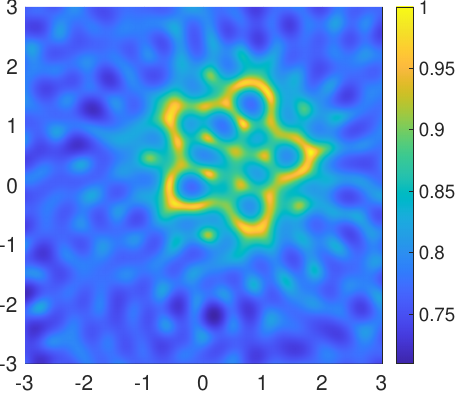}
      \caption{LSM, additive, $\delta=0.5$}
      \label{fig:5arms_k2pi_lsm_add05}
    \end{subfigure} \quad
    \begin{subfigure}[b]{0.3\linewidth}
      \centering
      \includegraphics[width=1\linewidth]{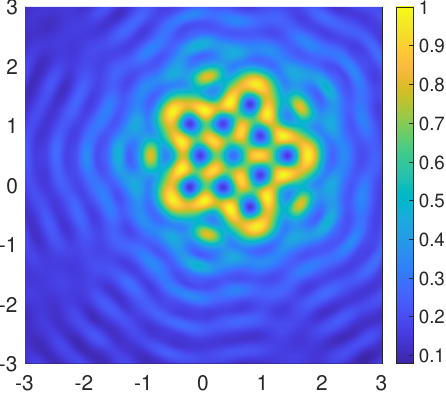}
      \caption{DSM--1, additive, $\delta=0.5$}
      \label{fig:5arms_k2pi_dsm1_add05}
    \end{subfigure}   \quad
    \begin{subfigure}[b]{0.3\linewidth}
      \centering
      \includegraphics[width=1\linewidth]{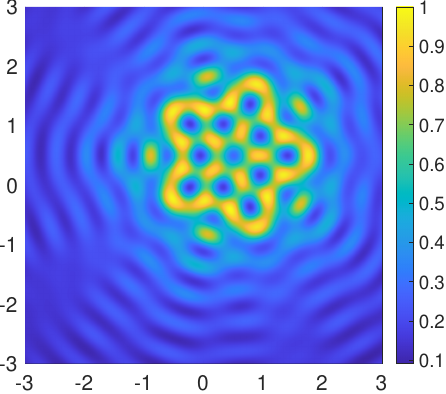}
      \caption{DSM--2, additive, $\delta=0.5$}
      \label{fig:5arms_k2pi_dsm2_add05}
    \end{subfigure} 

    \centering
    \begin{subfigure}[b]{0.3\linewidth}
      \centering
      \includegraphics[width=1\linewidth]{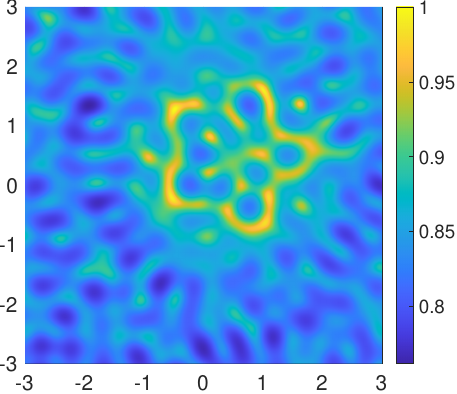}
      \caption{LSM, additive, $\delta=1.0$}
      \label{fig:5arms_k2pi_lsm_add1}
    \end{subfigure} \quad
    \begin{subfigure}[b]{0.3\linewidth}
      \centering
      \includegraphics[width=1\linewidth]{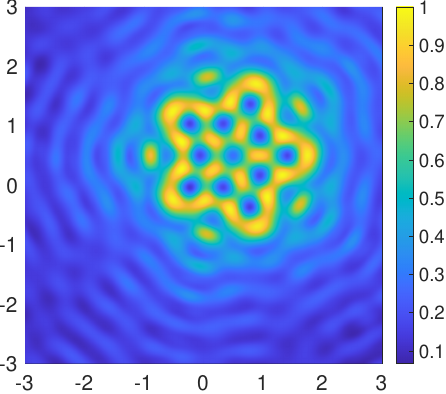}
      \caption{DSM--1, additive, $\delta=1.0$}
      \label{fig:5arms_k2pi_dsm1_add1}
    \end{subfigure}   \quad
    \begin{subfigure}[b]{0.3\linewidth}
      \centering
      \includegraphics[width=1\linewidth]{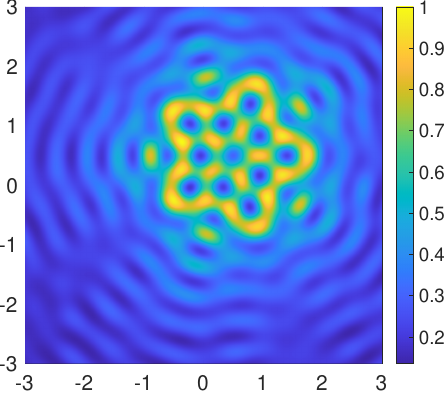}
      \caption{DSM--2, additive, $\delta=1.0$}
      \label{fig:5arms_k2pi_add1}
    \end{subfigure} 

\caption{Reconstructions using the LSM, DSM--1, and DSM--2 for the 5-arms domain at wavenumber $k=2\pi$ using data with 5\% (top row), 50\%(middle row), and 100\%(bottom row) {\bf additive} noise.}    
    \label{fig:ex2_rec2pi_add}
\end{figure}

\begin{figure}[!ht]
    \centering
    \begin{subfigure}[b]{0.3\linewidth}
      \centering
      \includegraphics[width=1\linewidth]{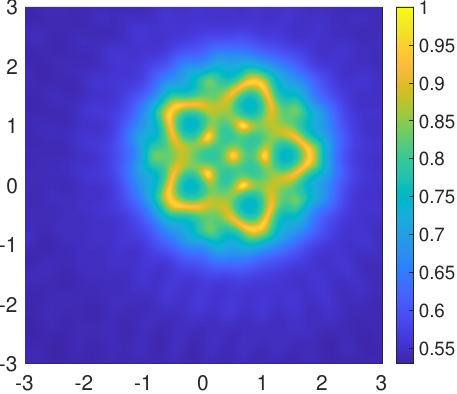}
      \caption{LSM, mult., $\delta=0.05$}
      \label{fig:5arms_k2pi_lsm_multi005}
    \end{subfigure} \quad
    \begin{subfigure}[b]{0.3\linewidth}
      \centering
      \includegraphics[width=1\linewidth]{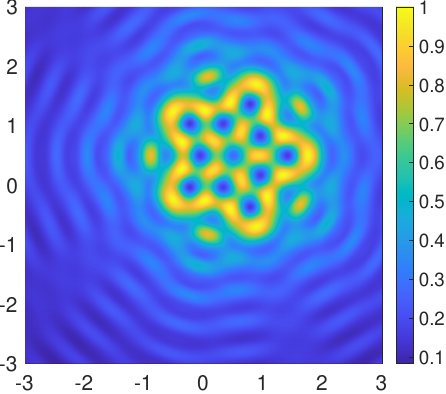}
      \caption{DSM--1, mult., $\delta=0.05$}
      \label{fig:5arms_k2pi_dsm1_multi005}
    \end{subfigure}   \quad
    \begin{subfigure}[b]{0.3\linewidth}
      \centering
      \includegraphics[width=1\linewidth]{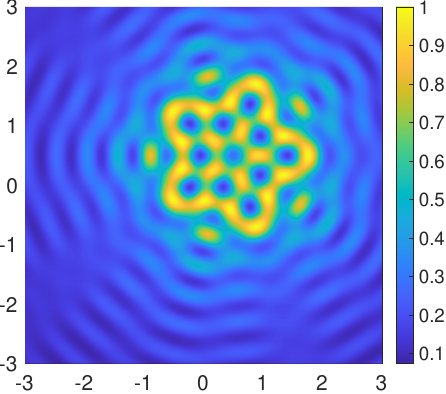}
      \caption{DSM--2, mult., $\delta=0.05$}
      \label{fig:5arms_k2pi_dsm2_multi005}
    \end{subfigure}        

    \centering
    \begin{subfigure}[b]{0.3\linewidth}
      \centering
      \includegraphics[width=1\linewidth]{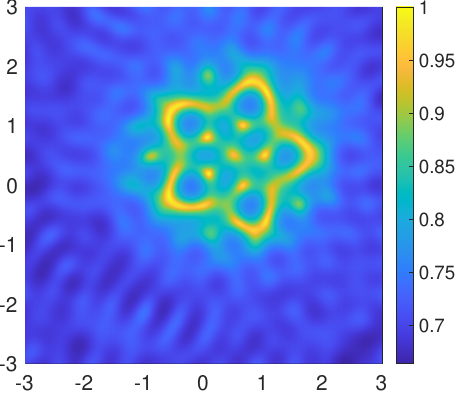}
      \caption{LSM, mult., $\delta=0.5$}
      \label{fig:5arms_k2pi_lsm_multi05}
    \end{subfigure} \quad
    \begin{subfigure}[b]{0.3\linewidth}
      \centering
      \includegraphics[width=1\linewidth]{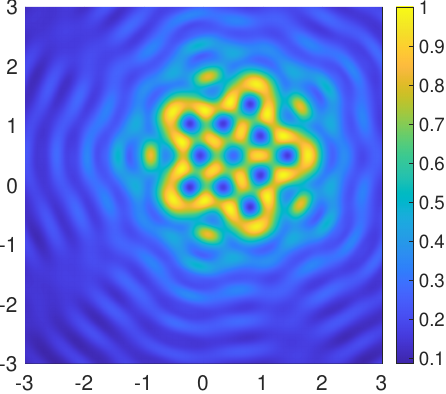}
      \caption{DSM--1, mult., $\delta=0.5$}
      \label{fig:5arms_k2pi_dsm1_multi05}
    \end{subfigure}   \quad
    \begin{subfigure}[b]{0.3\linewidth}
      \centering
      \includegraphics[width=1\linewidth]{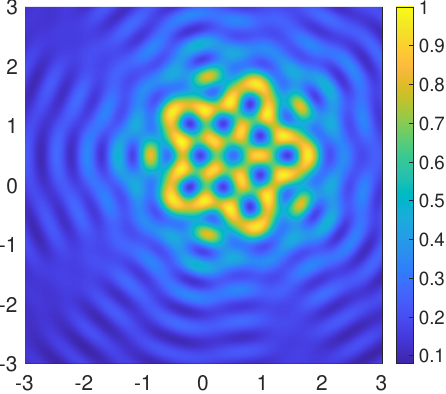}
      \caption{DSM--2, mult., $\delta=0.5$}
      \label{fig:5arms_k2pi_dsm2_multi05}
    \end{subfigure} 

    \centering
    \begin{subfigure}[b]{0.3\linewidth}
      \centering
      \includegraphics[width=1\linewidth]{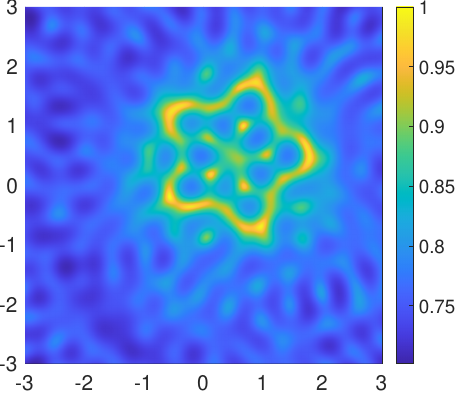}
      \caption{LSM, mult., $\delta=1.0$}
      \label{fig:5arms_k2pi_lsm_multi1}
    \end{subfigure} \quad
    \begin{subfigure}[b]{0.3\linewidth}
      \centering
      \includegraphics[width=1\linewidth]{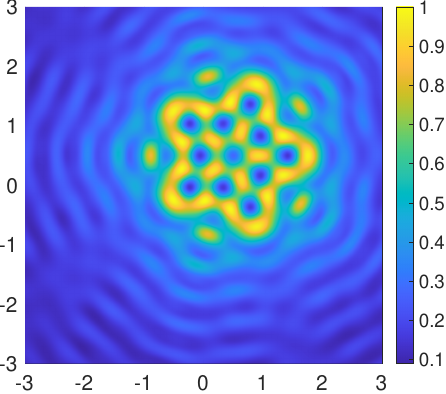}
      \caption{DSM--1, mult., $\delta=1.0$}
      \label{fig:5arms_k2pi_dsm1_multi1}
    \end{subfigure}   \quad
    \begin{subfigure}[b]{0.3\linewidth}
      \centering
      \includegraphics[width=1\linewidth]{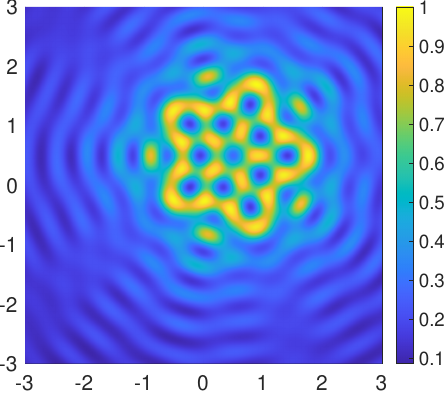}
      \caption{DSM--2, mult., $\delta=1.0$}
      \label{fig:5arms_k2pi_dsm2_multi1}
    \end{subfigure} 

\caption{Reconstructions using the LSM, DSM--1, and DSM--2 for the 5-arms domain at wavenumber $k=2\pi$ using data with 5\% (top row), 50\%(middle row), and 100\%(bottom row) {\bf multiplicative} noise.}    
    \label{fig:ex2_rec2pi_multi}
\end{figure}

\subsection{Example 3 -- Variable Poisson's Ratio}

In this experiment, we consider the effect of the Poisson's ratio by testing the sampling methods on problems with different values of $\nu$. We examine the behavior of the methods for $\nu = -0.5$, $0$, and $0.5$, which are approximately the Poisson's ratio for auxetic materials, cork, and rubber, respectively. The other examples in this paper use $\nu = 0.3$ which corresponds to ice and steel.

We use the 5-arms domain in this experiment. We consider two frequencies for the measurements, $k=2\pi$ and $k=20\pi$. For $k=2\pi$, we set $N_d=N_r=128$; for $k=20\pi$, we use $N_d=N_r=1024$. To avoid inverse crimes, we insert $5\%$ additive noise to the measured data. The indicator functions of the inverse methods are calculated on a uniform grid of 300 points in each direction over the domain $[-3,3]$. The regularization parameter for the LSM was chosen to be $\alpha = 10^{-1}$.

The results are presented in Figures \ref{fig:ex4_2pi_poisson} for $k=2\pi$ and \ref{fig:ex4_20pi_poisson} for $20\pi$. The DSM indicator functions do not appear to vary significantly with the Poisson's ratio. In contrast, the LSM indicator function shows a slight dependence on $\nu$, though the results are  similar. 

\begin{figure}[!ht]
    \centering
    \begin{subfigure}[b]{0.3\linewidth}
      \centering
      \includegraphics[width=1\linewidth]{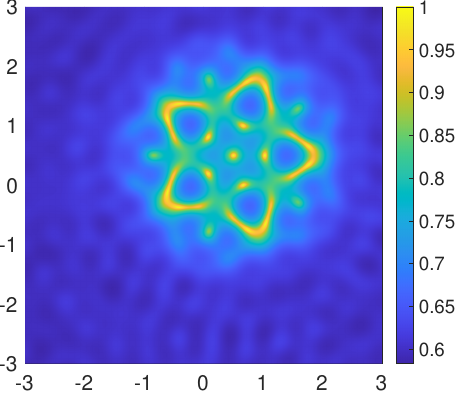}
      \caption{LSM, $k=2\pi$, $\nu = -0.5$}
      \label{fig:5arms_k2pi_lsm_pr-0.5}
    \end{subfigure} \quad
    \begin{subfigure}[b]{0.3\linewidth}
      \centering
      \includegraphics[width=1\linewidth]{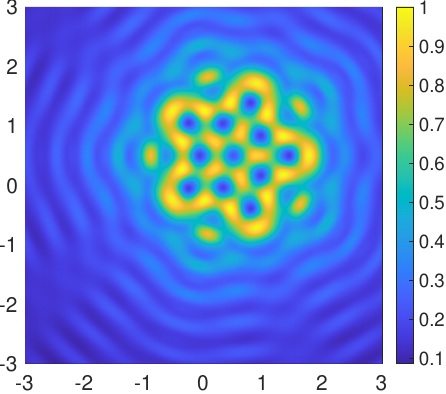}
      \caption{DSM--1, $k=2\pi$, $\nu = -0.5$}
      \label{fig:5arms_k2pi_dsm1_pr-0.5}
    \end{subfigure}   \quad
    \begin{subfigure}[b]{0.3\linewidth}
      \centering
      \includegraphics[width=1\linewidth]{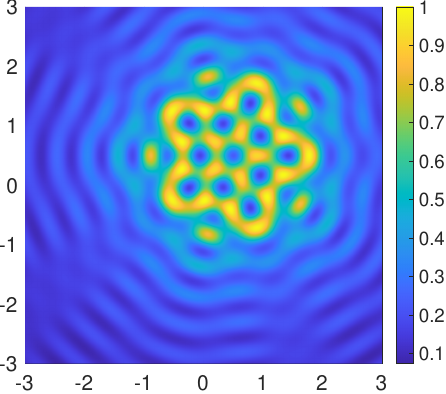}
      \caption{DSM--2, $k=2\pi$, $\nu = -0.5$}
      \label{fig:5arms_k2pi_dsm2_pr-0.5}
    \end{subfigure}        

    \centering
    \begin{subfigure}[b]{0.3\linewidth}
      \centering
      \includegraphics[width=1\linewidth]{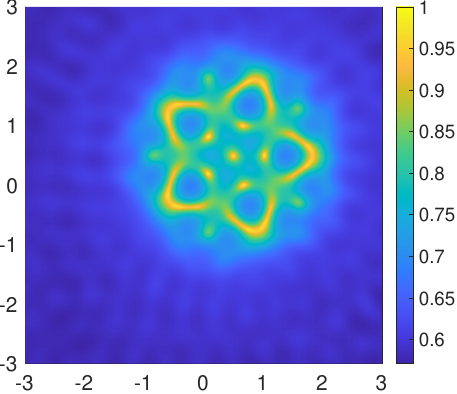}
      \caption{LSM, $k=2\pi$, $\nu = 0$}
      \label{fig:5arms_k2pi_lsm_pr0}
    \end{subfigure} \quad
    \begin{subfigure}[b]{0.3\linewidth}
      \centering
      \includegraphics[width=1\linewidth]{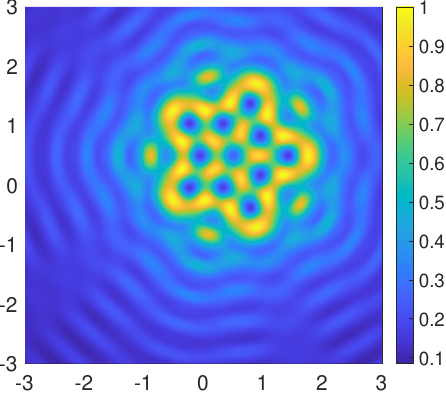}
      \caption{DSM--1, $k=2\pi$, $\nu = 0$}
      \label{fig:5arms_k2pi_dsm1_pr0}
    \end{subfigure}   \quad
    \begin{subfigure}[b]{0.3\linewidth}
      \centering
      \includegraphics[width=1\linewidth]{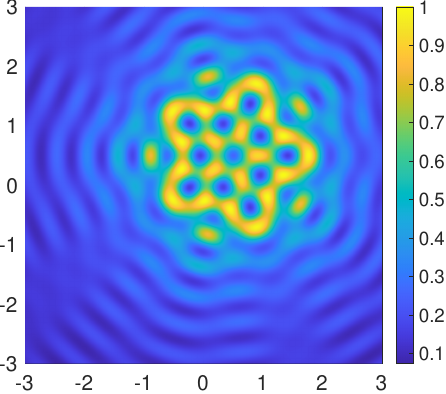}
      \caption{DSM--2, $k=2\pi$, $\nu = 0$}
      \label{fig:5arms_k2pi_dsm2_pr0}
    \end{subfigure} 

    \centering
    \begin{subfigure}[b]{0.3\linewidth}
      \centering
      \includegraphics[width=1\linewidth]{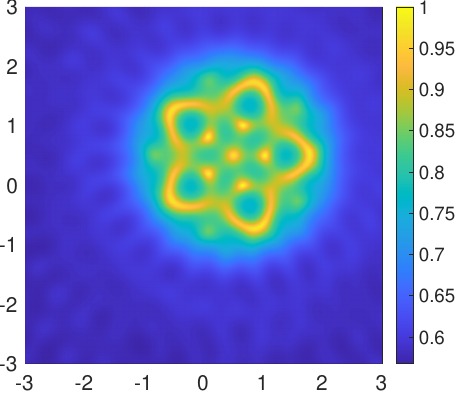}
      \caption{LSM, $k=2\pi$, $\nu = 0.5$}
      \label{fig:5arms_k2pi_lsm_pr0.5}
    \end{subfigure} \quad
    \begin{subfigure}[b]{0.3\linewidth}
      \centering
      \includegraphics[width=1\linewidth]{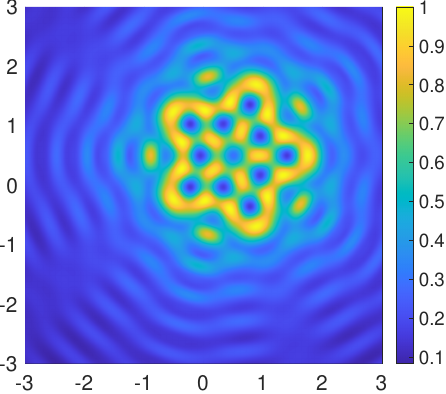}
      \caption{DSM--1, $k=2\pi$, $\nu = 0.5$}
      \label{fig:5arms_k2pi_dsm1_pr0.5}
    \end{subfigure}   \quad
    \begin{subfigure}[b]{0.3\linewidth}
      \centering
      \includegraphics[width=1\linewidth]{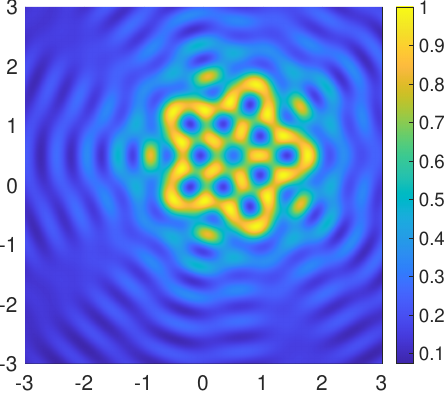}
      \caption{DSM--2, $k=2\pi$, $\nu = 0.5$}
      \label{fig:5arms_k2pi_dsm2_pr0.5}
    \end{subfigure} 

\caption{Reconstructions using the LSM, DSM--1, and DSM--2  for the 5-arms domain at wavenumber $k=2\pi$ with Poisson's ratio $\nu=-0.5$ (top row), $0$ (middle row), and $0.5$ (bottom row).}    
    \label{fig:ex4_2pi_poisson}
\end{figure}

\begin{figure}[!ht]
    \centering
    \begin{subfigure}[b]{0.3\linewidth}
      \centering
      \includegraphics[width=1\linewidth]{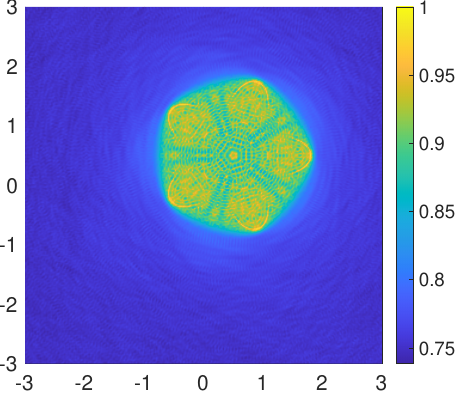}
      \caption{LSM, $k=20\pi$, $\nu = -0.5$}
      \label{fig:5arms_k20pi_lsm_pr-0.5}
    \end{subfigure} \quad
    \begin{subfigure}[b]{0.3\linewidth}
      \centering
      \includegraphics[width=1\linewidth]{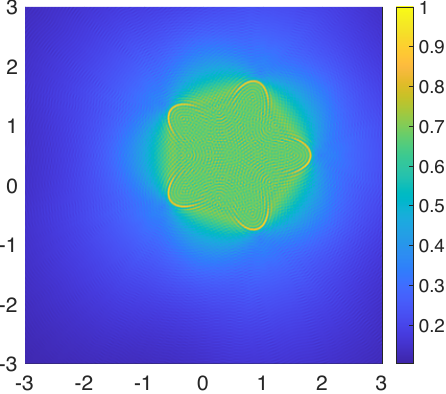}
      \caption{DSM--1, $k=20\pi$, $\nu = -0.5$}
      \label{fig:5arms_k02pi_dsm1_pr-0.5}
    \end{subfigure}   \quad
    \begin{subfigure}[b]{0.3\linewidth}
      \centering
      \includegraphics[width=1\linewidth]{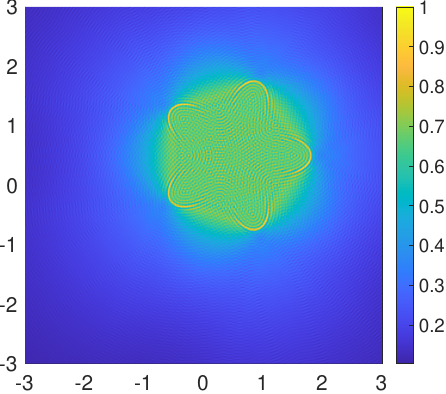}
      \caption{DSM--2, $k=20\pi$, $\nu = -0.5$}
      \label{fig:5arms_k20pi_dsm2_pr-0.5}
    \end{subfigure}        

    \centering
    \begin{subfigure}[b]{0.3\linewidth}
      \centering
      \includegraphics[width=1\linewidth]{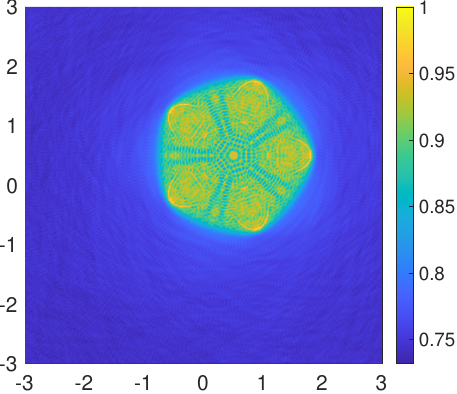}
      \caption{LSM, $k=20\pi$, $\nu = 0$}
      \label{fig:5arms_k20pi_lsm_pr0}
    \end{subfigure} \quad
    \begin{subfigure}[b]{0.3\linewidth}
      \centering
      \includegraphics[width=1\linewidth]{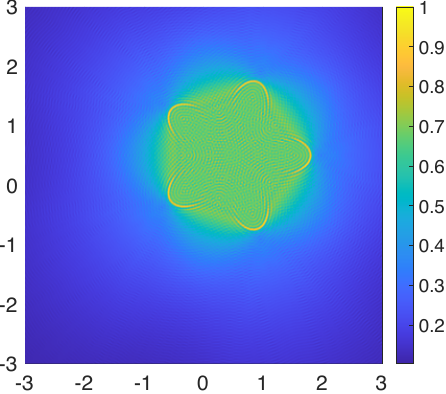}
      \caption{DSM--1, $k=20\pi$, $\nu = 0$}
      \label{fig:5arms_k20pi_dsm1_pr0}
    \end{subfigure}   \quad
    \begin{subfigure}[b]{0.3\linewidth}
      \centering
      \includegraphics[width=1\linewidth]{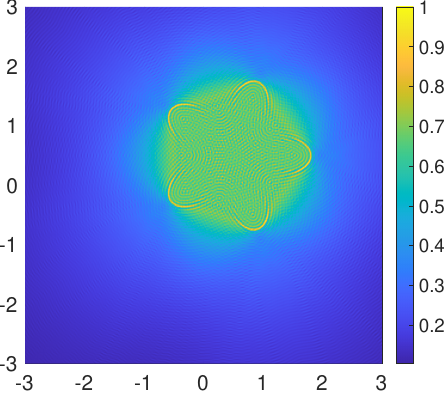}
      \caption{DSM--2, $k=20\pi$, $\nu = 0$}
      \label{fig:5arms_k20pi_dsm2_pr0}
    \end{subfigure} 

    \centering
    \begin{subfigure}[b]{0.3\linewidth}
      \centering
      \includegraphics[width=1\linewidth]{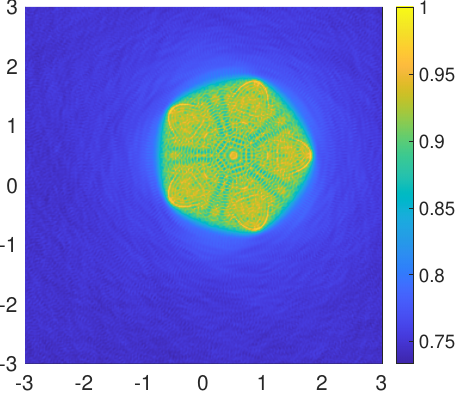}
      \caption{LSM, $k=20\pi$, $\nu = 0.5$}
      \label{fig:5arms_k20pi_lsm_pr0.5}
    \end{subfigure} \quad
    \begin{subfigure}[b]{0.3\linewidth}
      \centering
      \includegraphics[width=1\linewidth]{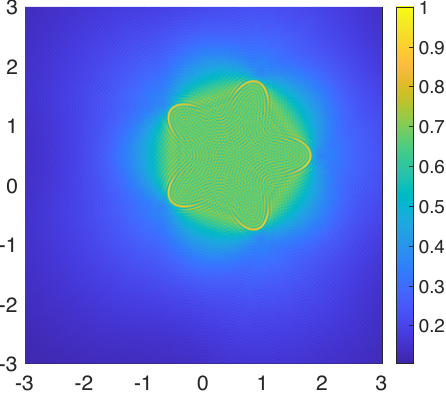}
      \caption{DSM--1, $k=20\pi$, $\nu = 0.5$}
      \label{fig:5arms_k20pi_dsm1_pr0.5}
    \end{subfigure}   \quad
    \begin{subfigure}[b]{0.3\linewidth}
      \centering
      \includegraphics[width=1\linewidth]{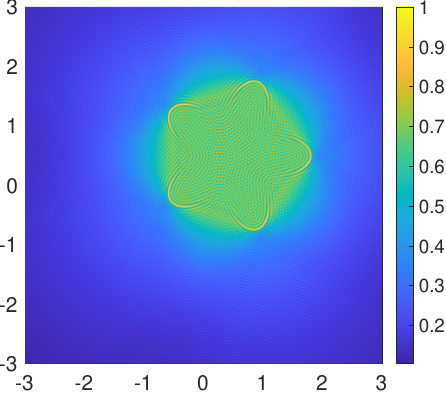}
      \caption{DSM--2, $k=20\pi$, $\nu = 0.5$}
      \label{fig:5arms_k20pi_dsm2_pr0.5}
    \end{subfigure} 

\caption{Reconstructions using the LSM, DSM--1, and DSM--2  for the 5-arms domain at wavenumber $k=20\pi$ with Poisson's ratio $\nu=-0.5$ (top row), $0$ (middle row), and $0.5$ (bottom row).}    
    \label{fig:ex4_20pi_poisson}
\end{figure}

\subsection{Example 4 -- Multiple obstacles}

One of the major advantages of sampling methods is their ability to handle problems involving multiple obstacles without requiring any prior information about the underlying configuration. In this example, the domain consists of three 5-arms obstacles centered at $(2,2.5)$, $(2,-2.5)$, and $(-2,0)$.

The incident waves have wavenumbers $k=2\pi$ and $k=10\pi$. For $k=2\pi$, we use $N_d=N_r=128$, and for $k=10\pi$, we use $N_d=N_r=512$. We add $5\%$ additive noise to the data to avoid inverse crimes. The indicator functions for the reconstruction methods are calculated on a uniform grid over the domain $[-10,10]$ with $500$ uniformly spaced points along each axis. As in the other examples, for the LSM, we use the regularization parameter $\alpha = 10^{-1}$. 

The results are shown in Figure \ref{fig:ex4_multi}. All methods produce good approximations of the obstacles, with the DSM yielding slightly more sharply defined reconstructions than the LSM. Note that recovering fine details between the obstacles is especially challenging.

\begin{figure}[!ht]
    \centering
    \begin{subfigure}[b]{0.3\linewidth}
      \centering
      \includegraphics[width=1\linewidth]{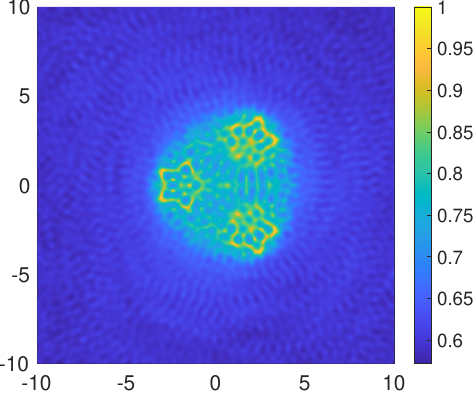}
      \caption{LSM, $k=2\pi$}
      \label{fig:multi_k2pi_lsm_add005}
    \end{subfigure} \quad
    \begin{subfigure}[b]{0.3\linewidth}
      \centering
      \includegraphics[width=1\linewidth]{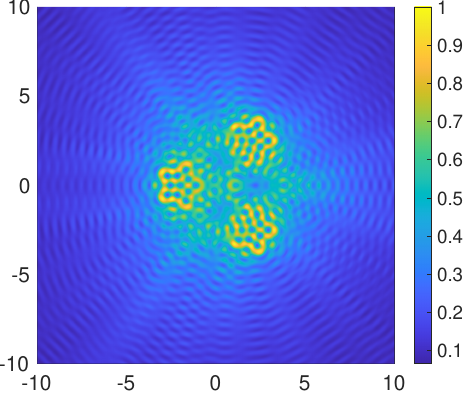}
      \caption{DSM--1, $k=2\pi$}
      \label{fig:multi_k2pi_dsm1_add005}
    \end{subfigure}   \quad
    \begin{subfigure}[b]{0.3\linewidth}
      \centering
      \includegraphics[width=1\linewidth]{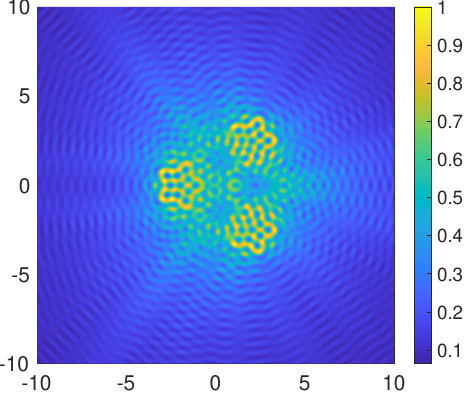}
      \caption{DSM--2, $k=2\pi$}
      \label{fig:multi_k2pi_dsm2_add005}
    \end{subfigure}        

   \centering
    \begin{subfigure}[b]{0.3\linewidth}
      \centering
      \includegraphics[width=1\linewidth]{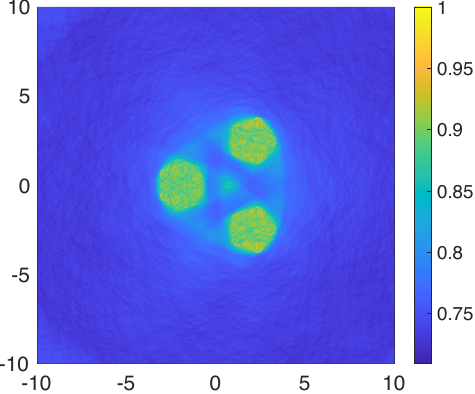}
      \caption{LSM, $k=10\pi$}
      \label{fig:multi_k10pi_lsm_add005}
    \end{subfigure} \quad
    \begin{subfigure}[b]{0.3\linewidth}
      \centering
      \includegraphics[width=1\linewidth]{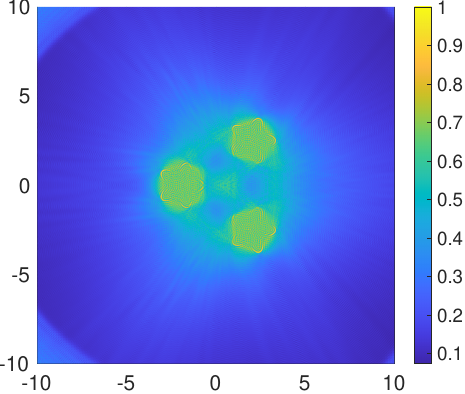}
      \caption{DSM--1, $k=10\pi$}
      \label{fig:multi_k10pi_dsm1_add005}
    \end{subfigure}   \quad
    \begin{subfigure}[b]{0.3\linewidth}
      \centering
      \includegraphics[width=1\linewidth]{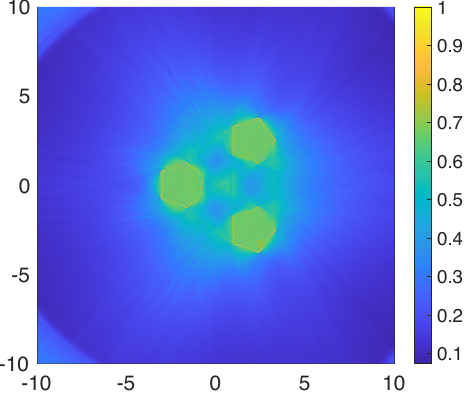}
      \caption{DSM--2, $k=10\pi$}
      \label{fig:multi_k10pi_dsm2_add005}
    \end{subfigure}     

\caption{Reconstructions using the LSM, DSM--1, and DSM--2 for the multiple-obstacle example at wavenumbers $k=2\pi$ (top row) and $k=10\pi$ (bottom row).}
    \label{fig:ex4_multi}
\end{figure}

\subsection{Example 5 -- Limited Data}

Finally, we report how the amount of available data affects the reconstructions. In our experiments, we observed that when the data are insufficient--particularly for the LSM--the quality of the reconstructions deteriorates. We consider two scenarios: a single 5-arms domain, and a multiple-obstacle configuration consisting of three 5-arms obstacles.

For the single obstacle case, we take $k=20\pi$, while for the multiple-obstacle case we use $k=10\pi$. In both settings, we use $N_d=N_r=128$. The indicator functions for the methods are calculated on the domain $[-3,3]^2$ with $300$ points per axis for the single‑obstacle case, and on the domain $[-10,10]^2$ with $500$ points per axis for the multiple‑obstacle case. The regularization parameter is chosen empirically as $10^{-1}$. Additive noise of $5\%$ is added to the measured data to avoid inverse crimes.

The results are shown in Figure \ref{fig:ex5_limited_data}. The upper row presents the reconstructions for the single‑obstacle experiment, while the bottom row displays the results for multiple obstacles. Note that the LSM reconstruction deteriorates more significantly than those obtained with the DSM when the amount of data is limited. In particular, for the multiple‑obstacle case, the LSM reconstruction becomes unclear to the point where even the number of obstacles is difficult to identify.

\begin{figure}[!ht]
    \centering
    \begin{subfigure}[b]{0.3\linewidth}
      \centering
      \includegraphics[width=1\linewidth]{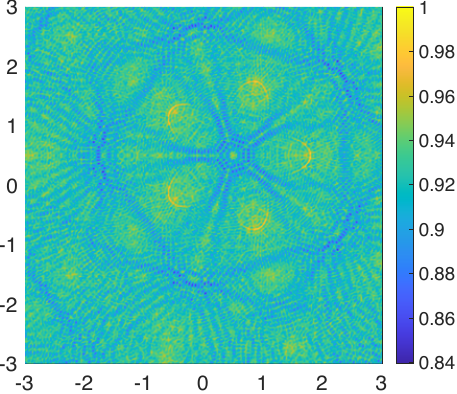}
      \caption{LSM, single obstacle}
      \label{fig:5arms_k20pi_lsm_data128}
    \end{subfigure} \quad
    \begin{subfigure}[b]{0.3\linewidth}
      \centering
      \includegraphics[width=1\linewidth]{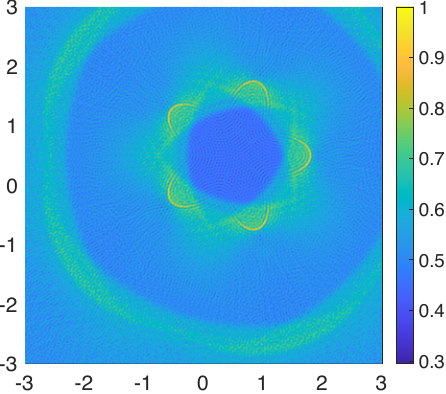}
      \caption{DSM--1, single obstacle}
      \label{fig:5arms_k20pi_dsm1_data128}
    \end{subfigure}   \quad
    \begin{subfigure}[b]{0.3\linewidth}
      \centering
      \includegraphics[width=1\linewidth]{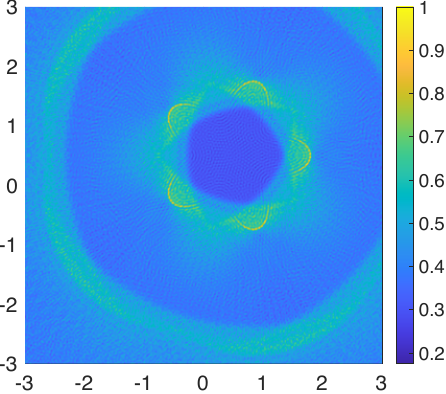}
      \caption{DSM--2, single obstacle}
      \label{fig:5arms_k20pi_dsm2_data128}
    \end{subfigure}        

   \centering
    \begin{subfigure}[b]{0.3\linewidth}
      \centering
      \includegraphics[width=1\linewidth]{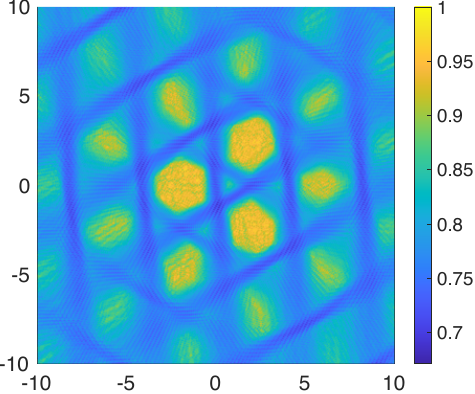}
      \caption{LSM, multiple obstacle}
      \label{fig:multi_k10pi_lsm_data128}
    \end{subfigure} \quad
    \begin{subfigure}[b]{0.3\linewidth}
      \centering
      \includegraphics[width=1\linewidth]{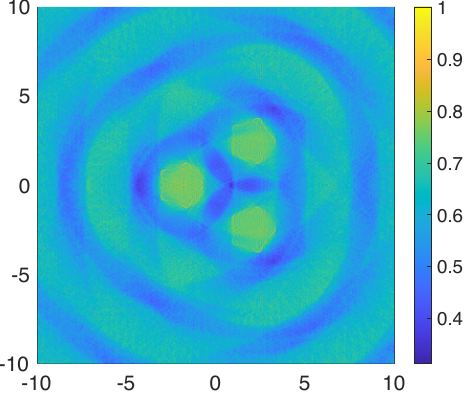}
      \caption{DSM--1, multiple obstacle}
      \label{fig:multi_k10pi_dsm1_data128}
    \end{subfigure}   \quad
    \begin{subfigure}[b]{0.3\linewidth}
      \centering
      \includegraphics[width=1\linewidth]{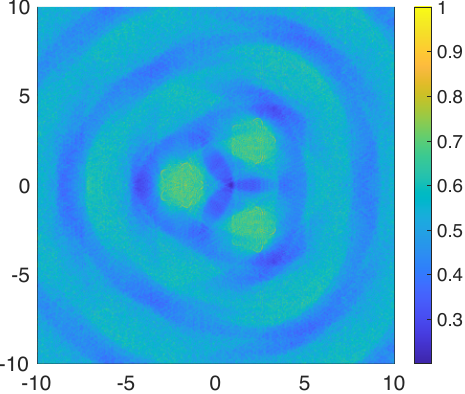}
      \caption{DSM--2, multiple obstacle}
      \label{fig:multi_k10pi_dsm2_data128}
    \end{subfigure}     

\caption{Reconstructions using the LSM, DSM--1, and DSM--2 with a limited amount of data ($N_d=N_r=128$). Top: incident wave $k=20\pi$ and a single 5-arms domain. Bottom: incident wave $k=10\pi$ and the three‑obstacle configuration.}    
    \label{fig:ex5_limited_data}
\end{figure}


\section{Conclusions} \label{sec:conclusions}

We have developed a theoretical foundation for solving the biharmonic inverse scattering with supported plate boundary conditions using the linear sampling method (LSM) and the direct sampling method (DSM). To this end, we have established and characterized indicator functions that allow the recovery of impenetrable supported cavities. A comprehensive set of numerical experiments were used to examine the influence of frequency, noise, Poisson's ratio, amount of data, and number of obstacles.

The results show that both LSM and DSM reliably recover the obstacle’s position and its convex hull, with DSM offering slightly better robustness--especially under noise or limited data availability. Consistent with related findings in Helmholtz, neither method resolves small cavities, reflecting intrinsic limitations of qualitative far‑field approaches. LSM additionally requires careful regularization and is more sensitive to noise and data sparsity, though its performance remains comparable in well‑resolved scenarios. For multiple obstacles, both methods correctly identify the configuration, with DSM again maintaining stability under reduced data. 

These observations highlight the strengths and inherent limitations of sampling‑based techniques for biharmonic scattering. From a theoretical standpoint, we hope to extend the existing integral equation formulation for the supported plate to be able to handle spurious resonances and establish the invertibility of the forward operator for real frequencies. We also hope to extend the theory of sampling methods to the free plate boundary conditions. Lastly, we are interested in using these methods to develop boundary extraction heuristics and explore direct reconstruction methods based on neural networks as in \cite{zhou2023neural}. \\


\noindent{\bf Acknowledgments:} 
The authors thank Jeremy Hoskins, Tristan Goodwill, Adriaan de Clercq, and Isaac Harris for generous and insightful discussions. \\

\noindent{\bf Data availability statement:} 
 Data and code supporting the findings of this study are available upon reasonable request from the authors.

\bibliography{bib_file}

\end{document}